\documentclass[11pt,a4paper,twoside]{amsart}
\usepackage{amsmath,amssymb,amsfonts,amsthm}
\usepackage{multirow}
\usepackage{xcolor}
\usepackage{hyperref}
\usepackage{graphicx}

\newcommand{\cg}{{\mathfrak g}}

\def\sgrad{\textrm{sgrad}}

\def\eps{\epsilon}

\newcommand{\cD}{{\mathcal D}}
\newcommand{\cL}{{\mathcal L}}

\def\cg{{\mathfrak g}}
\def\C{{\mathbb C}}
\def\Sph{{\mathbb S}}
\def\L{\bigtriangleup}

\def\<{\langle}
\def\>{\rangle}

\def \D{{\mathcal{D}}}

\def\N{\mathbb{N}}

\newtheorem{The}{Theorem}[section]
\newtheorem{prop}[The]{Proposition}
\newtheorem{Lem}[The]{Lemma}
\newtheorem{cor}[The]{Corollary}
\theoremstyle{definition}
\newtheorem{Def}[The]{Definition}

\newtheorem{Rem}[The]{Remark}

\thanks{2010 Mathematical Subject Classification.
 58D05, 35Q35, 53C22, 53C80.}

\begin{document}
\title{Conjugate points along spherical harmonics }
\author{ Ali Suri }
\address{Ali Suri, Universit\"{a}t Paderborn, Warburger Str. 100,
33098 Paderborn, Germany}
\email{asuri@math.upb.de}
\maketitle {\hspace{2.5cm}}

\begin{abstract} 
Utilizing structure constants, we present a version of the Misiolek criterion for identifying conjugate points. We propose an approach that enables us to locate these points along solutions of the quasi-geostrophic equations on the sphere $\Sph^2$. We demonstrate that for any spherical harmonics $Y_{lm}$ with $1 \leq |m| \leq l$, except for $Y_{1\pm1}$ and $Y_{2\pm 1}$, conjugate points can be determined along the solution generated by the velocity field $e_{lm}=\nabla^\perp Y_{lm}$. Subsequently, we investigate the impact of the Coriolis force on the occurrence of conjugate points. Moreover, for any zonal flow generated by the velocity field $\nabla^\perp Y_{l_1~0}$, we demonstrate that varying the rotation rate can lead to the appearance of conjugate points along the corresponding solution, where $l_1 = 2k+1. \in \mathbb{N}$ Additionally, we prove the existence of conjugate points along (complex) Rossby-Haurwitz waves and explore the effect of the Coriolis force on their stability.

\textbf{Keywords}: Conjugate points, Group of volume preserving diffeomorphisms, Misiolek criterion, Spherical harmonics, structure constants, quasi-geostrophic equations, zonal flow, Coriolis force, central extension.\\

%%%%%%%%%%%%%%%%%%%%%%%%%%%%%%%%%%%%%%%%%%%%%%%%%%%%%%%%%%%%%%%%%%%%
\end{abstract}

\pagestyle{headings} \markright{ Conjugate points along spherical harmonics}
\tableofcontents
%
%
%%%%%%%%%%%%%%%%%%%%%%%%%%%%%%%%%%%%%%%          Introduction       %%%%%%%%%%%%%%%%%%%%%%%%%%%%%%%%%%%%%%%%
%%%%%%%%%%%%%%%%%%%%%%%%%%%%%%%%%%%%%%%          Introduction       %%%%%%%%%%%%%%%%%%%%%%%%%%%%%%%%%%%%%%%%
%
%

\section{Introduction}
Let $M$ be a compact Riemannian manifold filled with  an incompressible non-viscous fluid. The group of volume-preserving diffeomorphisms, denoted by ${\mathcal{D}}_{vol}(M)$, serves as the configuration space (group) for this motion. In his work, Vladimir Arnold \cite{Arnold} considered  the $\mathcal{L}^2$ kinetic energy on the Lie algebra $\cg=T_e {\mathcal{D}}_{vol}(M)$ of divergence-free vector fields and extended it to a right-invariant metric on this group. He observed that the geodesic equations of  this $\mathcal{L}^2$ metric correspond to solutions of the Euler equations for incompressible fluids . The analytical features of this method were further developed by Ebin and Marsden \cite{Ebin-Marsden}.

Subsequently, this approach has been employed to study numerous nonlinear partial differential equations in mathematical physics. Indeed, altering the manifold $M$, the configuration space (group), and the weak Riemannian metric provides a way to examine group-theoretic geometric mechanics in diverse situations (for a survey see \cite{Arnold98}). Equations that appear within this framework are referred to as "Euler-Arnold" equations.

For $M=\mathbb{T}^2$ the two-dimensional flat torus, Arnold computed the sectional curvature of ${\mathcal{D}}_{vol}(\mathbb{T}^2)$, finding negativity in most directions. This suggests that nearby geodesics rapidly diverge, making the space of solutions unstable.
Another interpretation of this phenomenon is the unreliability of long-term weather prediction, as investigated by \cite{Arnold}, \cite{Luk,Yoshida}, and \cite{Suri-2023} in the context of a perfect fluid on a flat torus, a sphere and a rotating sphere respectively. It is observed that rotation could have a stabilizing effect on fluid motion \cite{Lee-Pre,Suri-2023}.

Arnold raised the question of the existence of conjugate points after noticing that on ${\mathcal{D}}_{vol}(\mathbb{T}^2)$, in certain directions, the sectional curvature is positive. %On the other hand, conjugate points, associated with critical points of the exponential map on ${\mathcal{D}}{vol}(M)$, can be interpreted as the Lagrangian stability of solutions \cite{Misiolek}. 
Misiolek in \cite{Misiolek} found conjugate points along rotation on $\Sph^2$ and $\Sph^3$ and in \cite{Misiolek99} on the flat torus $\mathbb{T}^2$. In \cite{Misiolek99} Misiolek provided a sufficient  criterion which guarantees the existence of conjugate points along time independent solution of the Euler-Arnold equation on ${\mathcal{D}}_{vol}(M)$. This criterion has been used to study the existence of conjugate points along time dependent and time independent solution on torus, sphere and ellipsoid \cite{Benn2021, Tauchi-Yoneda, Brigant-Preston, Drivas}.

In \cite{Benn2021}, Benn examined Rossby-Haurwitz waves on the sphere and demonstrated the existence of conjugate points along  these waves with some specific wave numbers. Apart from this study, very little is known about conjugate points along non-zonal solutions on the group of volume-preserving diffeomorphisms on the sphere. On the other hand, Tauchi and Yoneda proved that the Misiolek criterion (referred to as Misiolek curvature in their work) for zonal flows is consistently non-positive. This implies that the only zonal flow with conjugate points is generated by rotation (or equivalently by the first zonal spherical harmonic $Y_{1~0}$), as distinguished in \cite{Misiolek}.

In the case of the torus, the situation is more clear. Specifically, conjugate points along Kolmogorov flows $\psi_{mn}(x,y)=-\cos(mx)\cos(ny)$ which are eigenfunctions of the Laplacian on $\mathbb{T}^2$ exist for any $m, n \in \mathbb{N}$, with the exception of $(m, n) = (1, 1)$ \cite{Brigant-Preston}. Le Brigant and Preston \cite{Preston-Brig} proposed the problem of substituting the torus with the sphere and asserting the same findings.%They suggested to replace torus with sphere and state the same results.  

The positivity of the Misiolek criterion implies the positivity of curvature. More precisely, if there exist conjugate points along a geodesic, then the sectional curvature along this geodesic attains positive values. Due to this fact, we interpret the existence of conjugate points as an indicator of Lagrangian stability along the solution.

The presence  of the Coriolis effect makes the situation more practical. The only studies addressing the influence of the Coriolis effect on conjugate points and curvature are found in \cite{Lee-Pre}, \cite{Tauchi-Yoneda-2}, and \cite{Suri-2023}. In \cite{Tauchi-Yoneda-2},  Tauchi and Yoneda managed to establish the positivity of the Misiolek Criterion. However, providing an explicit solution for such a case remains an open problem.
%In \cite{Tauchi-Yoneda-2} the could prove the positivity of the Misiolek Criterion but, presenting an explicit solution %for such case remaind as an open problem.

On the other hand, in meteorology, the stability of zonal flows on a rotating sphere, according to the critical ratios for the rotation rate, has been studied by several authors, e.g., \cite{Bains} and \cite{Sasaki}.

\textbf{Contributions.} In this paper, we aim to address the conjectures mentioned earlier. First we prove that for any spherical harmonic $Y_{l_1m_1}$ with $1<m_1\leq l_1$ the Misiolek criterion $MC(e_{l_1m_1},e_{m-m})>0$ where $2\leq m\leq m_1$ and $e_{l_1m_1}=\nabla^\perp Y_{l_1m_1}$. The same holds true for $MC(e_{l_1~1},e_{~l_2~1})>0$ with $2\leq l_2<l_1$.
Moreover, we will show that, in the presence of the Coriolis force, the zonal flow $e_{l_10}$ ceases to be a global minimizer for any odd $l_1\in\mathbb{N}$ when the appropriate  speed and direction for rotation (rotation rates) are chosen.
Moreover, for (complex) Rossby-Haurwitz waves, which are time-dependent solutions of quasi-geostrophic equations, the existence of conjugate points is proved. 
We observe that the Coriolis effect stabilizes the system, generating conjugate points that wouldn’t appear without it.

\textbf{Outline.}
Section 2 is dedicated to a review of the geometry of the one-dimensional central extension of the quantomorphisms group and the derivation of the quasi-geostrophic equations. By employing the corresponding adjoint and co-adjoint operators, we present the appropriate Misiolek criterion for our framework. Then, following  \cite{Arak-Sav, Mess, Dowker, Mozheng}, we introduce spherical harmonics, Wigner $3j$ symbols, complex and real structure constants and discuss their properties.

In Section 3, after writing the Misiolek criterion according the the structure constants, we state the main theorem and prove that \textbf{i.} $MC(e_{l_1m_1},e_{m-m})>0$ for any $1< m_1\leq l_1$ and $2\leq m\leq m_1$ and \textbf{ii.} $MC(e_{l_1~1},e_{l_2~1})>0$ for any $2\leq l_2 < l_1$.
Moreover, we prove some linearity properties of the Misiolek criterion, which implies that by replacing $e_{m-m}$ with
$e_{m-m}+\sum_{j=1}^n x_j e_{l_jm_j}$, where $||(x_1,\dots,x_n)||$ is sufficiently small, conjugate points still exist.

In Section 4, first we introduce the Misiolek criterion in the presence of the Coriolis force, as determined by the structure constants.  
Computations demonstrate that the Coriolis effect introduces additional directions for conjugate points beyond those proposed in Section 3. In fact, using structure constants and their properties, we demonstrate that for odd $l_1$ and an appropriate choice of the rotation rate  $a$ (call this suitable choice the critical value), governing the speed and direction of rotation, conjugate points along solutions generated by the velocity field $e_{l_1 0}$ exist. A table indicating the critical values of $a$ for the wave numbers $l_1=3,5,7$ is presented.\\

On the other hand, Rossby-Haurwitz waves, widely used in meteorology, provide a time-dependent class of solutions for the quasi-geostrophic equations. We observe that the Coriolis effect stabilizes the system along these solutions, generating conjugate points that wouldn't appear without it.

%In the appendix, using a different  approach, we establish that for even $l_1$, the Misiolek criterion
%$MC(e_{l_1~0},g)\leq 0$ for any $g\in C^\infty(\Sph^2)$.

%\textbf{The paper is supported by a Maple program that enables us to calculate structure constants and the Misiolek criterion for given wave numbers.}
%
%
%
%%%%%%%%%%%%%%%%%%%%%%%%%%%%%%%%%%%%%%%%%%%%%%%%%%%%%%%%%   Hopf Fibration and qg eq   %%%%%%%%%%%%%%%%%%%%%%%%%%%%%%%%%%%%%%%
%%%%%%%%%%%%%%%%%%%%%%%%%%%%%%%%%%%%%%%%%%%%%%%%%%%%%%%%%   Hopf Fibration and qg eq   %%%%%%%%%%%%%%%%%%%%%%%%%%%%%%%%%%%%%%%
%%%%%%%%%%%%%%%%%%%%%%%%%%%%%%%%%%%%%%%%%%%%%%%%%%%%%%%%%   Hopf Fibration and qg eq   %%%%%%%%%%%%%%%%%%%%%%%%%%%%%%%%%%%%%%%
%
\section{Quasi-geostrophic equations and Misiolek criterion}
Ebin and Preston in their work \cite{Ebin-Preston}  derived the $\beta$-plane quasi-geostrophic equation, often abbreviated as QGS, as an Euler-Arnold equation. In this context, they employed the $\mathcal{L}^2$ metric on the quantomorphism group $\D_q(\Sph^3)$. Then they used the Hopf fibration and central extension to  derive QGS as an Euler-Arnold equation  on $\widehat{\D_{vol}}(\Sph^2)$ the central extension of $\D_{vol}(\Sph^2)$.

Following their approach, we will first provide a brief review of certain results from \cite{Ebin-Preston-arXiv} and \cite{Ebin-Preston}. This presentation may involve a different approach. Subsequently, we introduce the relevant Misiolek criterion by utilizing the corresponding adjoint and co-adjoint operators. 
Following \cite{Ebin-Preston, Ebin-Preston-arXiv} consider  a Boothby-Wang fibration $\pi:M\longrightarrow N$ where $M$ is contact manifold with the contact form $\theta$, the Reeb vector field $E$ and $N$ is symplectic manifold with the symplectic form $\omega$ and the property $\pi^*\omega=d\theta$. The following lemma is true for any contact manifold $M$ and as a special case for $M=\Sph^3$. The Riemannian metric on $M$ is denoted by $<,>$.

%Moreover suppose that $E_1$ and $\theta$ be the Reeb vector field and contact form on $\Sph^3$ respectively. %In fact the Hopf fibration is an example of Boothby-Wang fibration. Ebin and Preston in \cite{Ebin-Preston, %Ebin-Preston-arXiv} using the Riemannian geometry of the quantomorphism group of $\Sph^3$ and the Boothby-%Wang fibration proposed an approach to study the quasi-geographic equations on $\Sph^2$ as we review here.

For $s>\frac{dim M}{2}+1$, the quantomorphism group $\cD^s_q(M)=\{\eta\in\cD^s(M);~\eta^*\theta=\theta\}$ admits a smooth manifold structure (corollary 2.7 \cite{Ebin-Preston-arXiv}). The tangent space is
\begin{equation*}
\cg:=T_e \cD^s_q(M)=\{ S_\theta f;~f\in \mathcal{F}_{E}^{s+1}(M,\mathbb{R})\}
\end{equation*}
where $\mathcal{F}_{E}^{s+1}(M,\mathbb{R})=\{f:M\longrightarrow \mathbb{R};~ f\textrm{ is~} H^{s+1} \textrm{~and~} E(f)=0\}$ and the operator $S_\theta$ is defined by the following properties
\begin{equation*}
u=S_\theta f \quad \Longleftrightarrow  \quad \theta(u)=f~~ \textrm{and}~~i_ud\theta=-df.
\end{equation*}
%In this case $S_\theta f=fE_1   -  \frac{1}{2}(E_3f)E_2   +  \frac{1}{2}(E_2f)E_3$ and
%
%\begin{equation*}
%\Delta_\theta f=\alpha^2  -  \frac{1}{4}E_2^2  -  \frac{1}{4}E_3^2
%\end{equation*}
%where 
The contact Laplacian is $\Delta_\theta=S_\theta^* S_\theta$ where $S_\theta^*$ is the adjoint of the $S_\theta$ with respect to the right invariant $\cL^2$-metric induced by
\begin{equation*}
\ll S_\theta f, S_\theta g \gg=\int_{M}<S_\theta f,S_\theta g>d\mu.
\end{equation*}
on $\cD^s_q(M)$. For any $f,g\in\mathcal{F}_{E}^{s+1}(M,\mathbb{R})$ the contact Poisson bracket is defined by the relation $\{f,g\}=(S_\theta f) g$. In this case we have $S_\theta\{f,g\}=[S_\theta f,S_\theta g]$ which means that $S_\theta$ is a Lie algebra morphism.
%
%%%%%%%%%%%%%%%%%%%%%%%%%%%%%%%%%%%%%%%%%%%%%%%%%%%%%%%%%    begin  Lemma  ad^* in D_q      %%%%%%%%%%%%%%%%%%%%%%%%%%%%%%%%%%%%%%%%%%%%
%%%%%%%%%%%%%%%%%%%%%%%%%%%%%%%%%%%%%%%%%%%%%%%%%%%%%%%%%    begin  Lemma  ad^* in D_q      %%%%%%%%%%%%%%%%%%%%%%%%%%%%%%%%%%%%%%%%%%%%
%

\begin{Lem}
Let $s>\frac{dim M}{2}+1$ and $\cg=T_e\cD^s_q(M)$. Then $ad^*_u:\cg\rightarrow\cg$ is given by
\begin{equation}\label{eq ad* on quantomorphism}
ad^*_uv=S_\theta\Delta_\theta^{-1} \{f,\Delta_\theta g\}
\end{equation}
where $f,g\in\mathcal{F}_{E}^{s+1}(M,\mathbb{R})$ and  $u=S_\theta f,v=S_\theta g$.
\end{Lem}
\begin{proof}
For $u=S_\theta f$, $v=S_\theta g$ and $w=S_\theta h$ in $\cg$ we have
\begin{eqnarray*}
\ll ad^*_uv,w \gg_\cg   &=&   \int_M  < ad^*_uv,w > d\mu = \int_M < v,ad_uw > d\mu=-\int_M < v,[u,w] > d\mu\\
&=&   -\int_M < S_\theta g,[S_\theta f,S_\theta h] > d\mu  =  -\int_M < S_\theta g,S_\theta\{f, h\} > d\mu\\
& =&    -\int_M S_\theta^*S_\theta g\{f, h\}d\mu   =   -\int_M \Delta_\theta g\{f, h\}d\mu\\
&=&   -\int_M \{\Delta_\theta g,f\}hd\mu  =   -\int_M S_\theta^*S_\theta \Delta_\theta^{-1}\{\Delta_\theta g,f\}hd\mu\\
&=&   -\int_M   <  S_\theta \Delta_\theta^{-1}\{  \Delta_\theta g,f  \}   ,   S_\theta h  >  d\mu  = \ll S_\theta \Delta_\theta^{-1}\{f,\Delta_\theta g \}   ,   w \gg_\cg
\end{eqnarray*}
Since $w=S_\theta h\in\cg$ was arbitrary we get $ad^*_uv=S_\theta\Delta_\theta^{-1}\{f,\Delta_\theta g\}$.
\end{proof}
%
%
%%%%%%%%%%%%%%%%%%%%%%%%%%%%%%%%%%%%%%%%%%%%%%%%%%%%%%%%%       end Lemma      %%%%%%%%%%%%%%%%%%%%%%%%%%%%%%%%%%%%%%%%%%%%
%%%%%%%%%%%%%%%%%%%%%%%%%%%%%%%%%%%%%%%%%%%%%%%%%%%%%%%%%       end Lemma      %%%%%%%%%%%%%%%%%%%%%%%%%%%%%%%%%%%%%%%%%%%%
%
As a result the Euler-Arnold (geodesic) equation on $\cD^s_q(M)$ is given by
\begin{eqnarray*}
0=\partial_tu+ad^*_uu   &=&    \partial_t S_\theta f + ad^*_{S_\theta f}S_\theta f\\
&=&  \partial_t S_\theta f + S_\theta\Delta_\theta^{-1}\{f,\Delta_\theta f\}\\
&=& S_\theta \Big(  \partial_t   f +  \Delta_\theta^{-1} \{f,\Delta_\theta f\}  \Big)
\end{eqnarray*}
which implies that  $\partial_t   f +  \Delta_\theta^{-1} \{f,\Delta_\theta f\}=0$. Now we apply the contact Laplacian on both sides of the last equation and we get
\begin{equation}\label{eq geodesic eq on D_q with L^2 metric}
\partial_t\Delta_\theta f +  \{f,\Delta_\theta f\}=0.
\end{equation}
In \cite{Ebin-Preston-arXiv} theorem 4.1 for a different approach to this equation.
We recall that in Darboux coordinates $(x^1,\dots,x^n,y^1,\dots y^n,z)$ for $M$ we have $\theta=dz+\sum_{k=1}^nx^kdy^k$, $E=\partial_z$
\begin{equation*}
    S_\theta f=\sum_{k=1}^n \Big( -\frac{\partial f}{\partial y^k} \frac{\partial }{\partial x^k} +      \frac{\partial f}{\partial x^k} \frac{\partial }{\partial y^k}  \Big)
+\Big( f -   \sum_{k=1}^nx^k\frac{\partial f}{\partial x^k}\Big)  \frac{\partial }{\partial z}     
\end{equation*}
and 
\begin{equation*}
\Delta_\theta f= 2f  - \sum_{k=1}^n   \frac{\partial }{\partial x^k}   \Big(   \big(1+(x^k)^2 \big) 
\frac{\partial f}{\partial x^k} \Big)     
-    \sum_{k=1}^n  \frac{\partial }{\partial y^k} \frac{\partial f }{\partial y^k}      
\end{equation*}
%
%
%%%%%%%%%%%%%%%%%%%%%%%%%%%%%%%%%%%%%%%%%%%%%%%%%%%%%%%%%%%     Central extension    %%%%%%%%%%%%%%%%%%%%%%%%%%%%%%%%%%%%%%%%%
%%%%%%%%%%%%%%%%%%%%%%%%%%%%%%%%%%%%%%%%%%%%%%%%%%%%%%%%%%%%%     Central extension    %%%%%%%%%%%%%%%%%%%%%%%%%%%%%%%%%%%%%%%%%
%%%%%%%%%%%%%%%%%%%%%%%%%%%%%%%%%%%%%%%%%%%%%%%%%%%%%%%%%%%%%     Central extension    %%%%%%%%%%%%%%%%%%%%%%%%%%%%%%%%%%%%%%%%%
\subsection{Central extension of the quantomorphism group} 
Suppose that $M=\Sph^3$ and  
\begin{eqnarray*}
\pi:\Sph^3   &\longrightarrow&    \Sph^2
\end{eqnarray*} 
is the Hopf fibration. Following \cite{Ebin-Preston} and \cite{Ebin-Preston-arXiv} we consider the central extension of the Lie Algebra $\cg=T_e\cD^s_q(\Sph^3)$ with $\mathbb{R}$ which is denoted  by $\hat\cg=\cg\ltimes_\Omega\mathbb{R}$. For $u=S_\theta f$ and $v=S_\theta g$ in $\cg$ the map $ \Omega(u,v)=\int_{\Sph^3}\phi\{f,g\}d\mu$ and $\phi:M\longrightarrow\mathbb{R}$ is a known function (usually the distance form equator). Recall that for any $(u,a),(v,b)\in \cg\ltimes_\Omega\mathbb{R}$  the Lie bracket is defined by
\begin{equation*}
\big[  (u,a),(v,b)\big]=\big(  S_\theta\{f,g\}  ,  \Omega(u,v)  \big)
\end{equation*}
and the inner product is
\begin{equation}\label{Inner prod hat S3}
\ll  (u,a),(v,b)  \gg_{\hat{\cg}}  = \int_{M} <S_\theta f  ,  S_\theta g>d\mu  + ab.
\end{equation}
The operator $T:\cg\longrightarrow \cg$ defined by the relation $\ll Tu,v \gg=\Omega(u,v)$ is given by $T(S_\theta f)=S_\theta\Delta_\theta^{-1}\{\phi,f\}$ (for more details see \cite{Suri-2023}).
Moreover  $\hat{ad}_{(u,a)}^*:\hat{\cg}\longrightarrow \hat{\cg}$ is given by
$  \hat{ad}_{(u,a)}^* (v,b)=(ad^*_uv  - b Tu,0)\in\hat{\cg}$  and  for the curve $(u,a):(-\epsilon,\epsilon)\longrightarrow \hat{\cg}$  the Euler-Arnold  equation is given by
\begin{equation*}
\left\{ \begin{array}{ll}  \partial_tu  +  ad^*_uu  -  a(t)Tu  =  0   \\
\partial_ta(t)=0
\end{array}\right.
\end{equation*}
The second equation implies that $a(t)=a$ is constant and following the procedure for derivation of equation (\ref{eq geodesic eq on D_q with L^2 metric}) for
$u=S_\theta f$ we have
\begin{equation}\label{eq geo-eq of cent-ext pre}
\partial_t\Delta_\theta f  +  \{f,\Delta_\theta f\}  -  a\{\phi,f\}  =  \partial_t\Delta_\theta f  +  \{f,\Delta_\theta f  + a\phi\} =  0
\end{equation}
%
%%%%%%%%%%%%%%%%%%%%%%%%%%%%%%%%%%%%%%%%%%%%%%%%%%%%%%%%%%%%%%%%%%           Covariant derivative    %%%%%%%%%%%%%%%%%%%%%%%%%%%%%%%%%%%%%%%%%%%%%
%%%%%%%%%%%%%%%%%%%%%%%%%%%%%%%%%%%%%%%%%%%%%%%%%%%%%%%%%%%%%%%%%%           Covariant derivative    %%%%%%%%%%%%%%%%%%%%%%%%%%%%%%%%%%%%%%%%%%%%%

Finally we note that the covariant derivative on $(\cD^s_q(\Sph^3),\ll,\gg_{\cg})$ is given by
\begin{eqnarray*}
2\nabla_uv    &=&   -ad_uv  +ad^*_uv+ad^*_vu\\
&=&     S_\theta\{f,g\}   +    S_\theta\Delta_\theta^{-1}\{f,\Delta_\theta g\}   +   S_\theta\Delta_\theta^{-1}\{g,\Delta_\theta f\}
\end{eqnarray*}
and for $(\cD^s_q(\Sph^3)\ltimes_\Omega\mathbb{R},\ll,\gg_{\hat{\cg}})$ we have
\begin{eqnarray*}
2\hat\nabla_{(u,a)}(v,b)  &=&  - \hat{ad}_{(u,a)}(v,b)+\hat{ad}^*_{(u,a)}(v,b)+\hat{ad}^*_{(v,b)}(u,a)\\
&=&   \Big(    S_\theta\{f,g\}  , - \Omega(S_\theta f,S_\theta g) \Big)  +  \Big(   S_\theta\Delta_\theta^{-1}\{f,\Delta_\theta g\} - bTS_\theta f    ,   0   \Big) \\
&&    +  \Big(   S_\theta\Delta_\theta^{-1}\{g,\Delta_\theta f\} - aTS_\theta g    ,   0   \Big)
\end{eqnarray*}
On the other hand
\begin{eqnarray*}
\hat\nabla_{(u,a)}(v,b)  +  \hat\nabla_{(v,b)}(u,a)    &=&    \frac{1}{2}\Big( -\hat{ad}_{(u,a)}(v,b)   +   \hat{ad}_{(u,a)}^*(v,b)    +    \hat{ad}_{(v,b)}^*(u,a)\\
&&    -\hat{ad}_{(v,b)}(u,a)   +  \hat{ad}_{(v,b)}^*(u,a)  +  \hat{ad}_{(u,a)}^*(v,b)   \Big)\\
&=&   \hat{ad}_{(u,a)}^*(v,b)    +    \hat{ad}_{(v,b)}^*(u,a).
\end{eqnarray*}
The last two equations imply that
\begin{equation}\label{eq symmetric covariant derivative}
\hat\nabla_{(u,a)}(v,b)  +  \hat\nabla_{(v,b)}(u,a)    =  ( ad^*_uv  - bTu  + ad^*_vu  - aTv  ,  0)
\end{equation}
%
%%%%%%%%%%%%%%%%%%%%%%%%%%%%%%%%%%%%%%%%%%%%%%%%%%%%%%%%%%%%%%%%%     Remark    %%%%%%%%%%%%%%%%%%%%%%%%%%%%%%%%%%%%%%%%%%%%%%%%%
%%%%%%%%%%%%%%%%%%%%%%%%%%%%%%%%%%%%%%%%%%%%%%%%%%%%%%%%%%%%%%%%%     Remark    %%%%%%%%%%%%%%%%%%%%%%%%%%%%%%%%%%%%%%%%%%%%%%%%%
\begin{Rem}
Since for any $f\in\mathcal{F}_{E }^{s+1}(\Sph^3,\mathbb{R})$ the function $f$ is constant in the direction of the Reeb field (e.g., with respect to the last variable in the local chart of $\Sph^3$). Moreover  integration on $\Sph^3$ reduces to integration on the symplectic quotient $\Sph^2$. In this case also $\Delta_\theta=\alpha^2-\Delta$ where $\Delta$ is the usual Laplacian on $\Sph^2$. Moreover the Euler-Arnold  equations (\ref{eq geodesic eq on D_q with L^2 metric}) and (\ref{eq geo-eq of cent-ext pre}) are given by
\begin{equation}\label{eq geodesic eq on D_q with L^2 metric reduced}
\partial_t(\Delta f-\alpha^2 f) +  \{f,\Delta  f\}=0
\end{equation}
and
\begin{equation}\label{eq geodesic-equation of central extension reduced general}
\partial_t(\Delta f-\alpha^2 f) +  \{f,\Delta  f\}  - a\{f, \phi\}=0.
\end{equation}
respectively (compare with Corollary 5 of \cite{Lee-Pre}).

In the sequel, consider the following parametrization for $\Sph^2$
\begin{eqnarray*}
f:(-1,1)\times (0,2\pi)  &\longrightarrow & \Sph^2\subseteq\mathbb{R}^3\\
(\mu,\lambda)   &\longmapsto & (\sqrt{1-\mu^2}\sin\lambda  ,  \sqrt{1-\mu^2}\cos\lambda  ,  \mu)
\end{eqnarray*}

In the  $\beta$-plane approximation model  the function $\phi$ is locally given by $\phi(\lambda,\mu)=\mu$  and $\{f,g\}$ is the Poisson bracket which in Darboux coordinates resembles
\begin{equation}\label{poisson bracket preston}
\{f,g\}=     \frac{\partial f}{\partial \lambda}\frac{\partial g}{\partial \mu}  -  \frac{\partial f}{\partial \mu} \frac{\partial g}{\partial \lambda}
\end{equation}
Note that  a different sign convention for the Poisson bracket would not change our results. Finally we note that, we will deal with the equation (\ref{eq geodesic eq on D_q with L^2 metric reduced}) and the following equation which we will call  it  QGS equation or Euler equation at the presence of the \textbf{Coriolis force} (see also \cite{Skiba} for the case that $\alpha^2=0$.)
\begin{equation}\label{eq geodesic-equation of central extension reduced}
\partial_t(\Delta f-\alpha^2 f) +  \{f,\Delta  f-a\mu\}  =0.
\end{equation}
\end{Rem}
%
%
%
%
%%%%%%%%%%%%%%%%%%%%%%%%%%%%%%%%%%%%%%%%%%%%%%%%%%%%%%%%%%%%%   subsection    Misiolek Curvature     %%%%%%%%%%%%%%%%%%%%%%%%%%%%%%%%%%%%%%%%%%
%%%%%%%%%%%%%%%%%%%%%%%%%%%%%%%%%%%%%%%%%%%%%%%%%%%%%%%%%%%%%   subsection     Misiolek Curvature     %%%%%%%%%%%%%%%%%%%%%%%%%%%%%%%%%%%%%%%%%%
%%%%%%%%%%%%%%%%%%%%%%%%%%%%%%%%%%%%%%%%%%%%%%%%%%%%%%%%%%%%%   subsection     Misiolek Curvature     %%%%%%%%%%%%%%%%%%%%%%%%%%%%%%%%%%%%%%%%%%
%
%
%
%
\subsection{Misiolek criterion for quasi-geostrophic equations}\label{sub sect. Misiolek curvature}

In this section, we review the concept of conjugate points and the method introduced by Misiolek \cite{Misiolek99} to find them. This  method uses a criterion which is now known as the 'Misiolek criterion'.

For a compact manifold (here without boundary), we know that $\cD^s_{vol}(M)$ admits a smooth manifold structure modeled on a Banach space. As a result, the corresponding Euler-Arnold equation, which could be considered an ordinary differential equation, has (local) solutions, and the dependence of solutions on the initial data is differentiable. This implies that the corresponding exponential map can be defined from an open set $U \subseteq T_e \cD^s_{vol}(M)$ as follows
\[
\exp_e:U\subseteq T_e \cD^s_{vol}(M)\longrightarrow \cD^s_{vol}(M)\quad;\quad v_0\longmapsto \exp_e(u_0):=\eta(1)
\]
where $\eta:(-\eps,\eps)\longrightarrow \cD^s_{vol}(M)$ is the unique curve with $\eta(0)=e$ and $\dot\eta(0)=u_0.$ If $M$ is a 2-dimensional manifold then, $\exp_e$ is defined on the whole tangent space. In the other words, $\exp$ gives us information about the behaviour of solutions according to the initial values. Singularities of the map
\[
D\exp_e:T_0 T_e \cD^s_{vol}(M)\simeq  T_e \cD^s_{vol}(M)    \longrightarrow T_{\eta(t_0)} \cD^s_{vol}(M)
\]
are called conjugate points. Note that we have two types of conjugate points. When $Dexp_e(t_0u_0)$ is not injective we call the conjugate point mono-conjugate and in the case that $Dexp_e(t_0u_0)$ is not surjective we have epi-conjugate point. The method ontroduced by Misiolek, catches mono-conjugate points. However, in the two-dimensional case ($dim(M)=2$), the exponential is a nonlinear Fredholm operator of index zero, that is epi-conjugate and mono-conjugate coincide. 

Consider the variation of the geodesic $\eta$ defined by 
\begin{eqnarray*}
\eta(s,t):=\exp_e(t(u_0+sv_0)).    
\end{eqnarray*}%(-\eps,\eps)\times=
Locally a Jacobi field along $\eta$ looks like 
\begin{eqnarray*}
J(t):=D\exp_e(tu_0)tv_0    
\end{eqnarray*}
with the in-tial conditions $J(0)=0$ and $\dot{J}(0)=v_0$ and satisfies the Jacobi equation
\begin{equation}\label{Jacobi equation}
\nabla_{\dot\eta} \nabla_{\dot\eta} J  +  R(J,\dot\eta)\dot\eta=0.
\end{equation}
where $\nabla$ and $R$ are the corresponding covariant derivative and curvature of $\cD^s_{vol}(M).$ The Jacobi equation \eqref{Jacobi equation} is obtained by calculating the second variation of the energy functional.

An splitting  of the equation \eqref{Jacobi equation} is given by
\begin{eqnarray}\label{split Jac eq}
&&  \partial_t v  - ad_uv=w\\
&&  \partial_t w  + ad^*_uw + ad^*_wu=0  \nonumber
\end{eqnarray}
where $J(t)=v(t)\circ\eta(t)$ and  $\dot\eta(t)=u(t)\circ\eta(t)$ (e.g.see \cite{Preston-thesis}, chapter 4). 
Practically, the equation \eqref{split Jac eq} is a linearization of the Euler-Arnold equation, and a Jacobi field corresponds to a deviation of the geodesic $\eta$. Moreover, conjugate points are obtained by finding $T>0$ with $J(0)=J(T)=0$.

In order to prove the existence of  such $T$, usually we use the following index 
\begin{equation}\label{Index}
I_{0}^T(Y,Y):=\int_0^T   \Big(      \ll         \partial_tY + \nabla_{\dot\eta} Y ,   \partial_t Y+ \nabla_{\dot\eta} Y   \gg_\cg  
- \ll  R(Y,\dot\eta)\dot\eta,Y \gg_\cg  \Big)dt
\end{equation}
where $Y$ is a vector field along $\eta$. The previous equation can be obtained by multiplying \eqref{Jacobi equation} by $J$ and integrating from $0$ to $T$, and then replacing $J$ with an arbitrary field $Y$.

For $Y$ as above, with $Y(0)=Y(T)=0$, Misiolek (\cite{Misiolek99}, lemma 3) proved that if there are no points conjugate to $e=\eta(0)$ along $\eta(t)$ for $0 < t < T$, then $I_{0}^T(Y,Y)\geq 0$. Moreover, if $I_{0}^T(Y,Y)< 0$ then there exists $0<t_0\leq T$ such that $e=\eta(0)$ and $\eta(t_0)$ are conjugate.

Suppose that $v\in\cg=T_e\cD^s_{vol}(M)$ be an stationary solution of the Euler-Arnold equation. Consider the field $Y(t)=\psi(t)v(\eta(t))$ where $\psi:\mathbb{R}\longrightarrow\mathbb{R}$ is a smooth function. Note that basically, $Y$ is the right-invariant vector field generated by $v$ and multiplied by the scalar function $\psi$. In this case that the index \eqref{Index} takes the form
\begin{equation}\label{Index}
I_{0}^T(Y,Y):=\int_0^T   \Big(     \dot\psi^2(t)||v||_\cg^2    -   \psi^2(t)     \ll  \nabla_{u}[u,v] + \nabla_{[u,v]}u , v \gg_\cg \Big)dt
\end{equation}
For the stationary solutions  $v,w$, the  Misiolek criterion is defined by 
\begin{equation}\label{MC original}
MC(u,v):=  \ll  \nabla_{u}[u,v] + \nabla_{[u,v]}u , v \gg_\cg.
\end{equation}
Now, if $MC(u,v)>0$ then the conjugate points along $
\eta$ exist. More precisely, suppose there exist $\kappa>0$ such that $MC(u,v)>\kappa||v||_\cg^3>0$ and define
\[
t_0:=\frac{\pi}{\sqrt{\kappa||v||_\cg}},\quad \psi(t)=\sin(t\sqrt{\kappa||v||_\cg}).
\]
Then
\begin{eqnarray*}
I_{0}^{t_0}(Y,Y)  &=&  \int_0^{t_0}   \Big(     \dot\psi^2(t)||v||_\cg^2    -   \psi^2(t)    MC(u,v) \Big)dt\\
&<&    \int_0^{t_0}   \Big(     \dot\psi^2(t)||v||_\cg^2    -   \psi^2(t)    \kappa||v||_\cg^3 \Big)dt\\
&=&    \kappa||v||_\cg^3 \int_0^{t_0}   \Big(     \cos^2(t\sqrt{\kappa||v||_\cg})  - \sin^2(t\sqrt{\kappa||v||_\cg})   \Big)dt   =0
\end{eqnarray*}
that is $e$ and $\eta(t_0) $ are conjugate. 

Benn in \cite{Benn2021} developed this method for time-dependent solutions $u$. Tauchi and Yoneda \cite{Tauchi-Yoneda-2} proved that the above machinery works for manifolds $G$ (possibly infinite-dimensional) which have a topological group structure and a right-invariant metric. In this case, the Misiolek criterion (called Misiolek curvature in \cite{Tauchi-Yoneda-2}) is given by
\begin{eqnarray*}
{MC}\big( u,v \big)   
&=&  - \ll  [u,v] , [u,v]\gg   -   \ll  [[u,v],  v ],u \gg_\cg\\
&=&   \ll  ad_uv , [u,v]\gg   +   \ll  ad_{[u,v]}  v ,u \gg_\cg\\
&=&   \ll  v , ad_u^*[u,v]\gg   +   \ll    v ,ad_{[u,v]}^*u \gg_\cg\\
&=&   \ll   ad_u^*[u,v]    +       ad_{[u,v]}^*u   ,  v\gg_\cg
\end{eqnarray*}
where  $[,]$ is the Lie bracket on $\cg:=T_eG$ and $\ll  .~ ,~ .\gg_\cg$  is the inner product on $\cg$ which generates the right invariant metric on $G$. The above expression on $T_e\cD^s_{vol}(M)$ reduces to the original criterion \eqref{MC original}.

It is not difficult to see that 
\[
MC(u,v)= \ll R(v,u)u,v\gg_\cg - ||\nabla_uv||_\cg^2
\]   
As a result positivity of the Misiolek criterion  implies that the (sectional) curvature is also positive. 

Intuitively,  existence of conjugate points means that if we perturb the geodesic generated by $u$ in the direction of $v$, the perturbed geodesic remains infinitesimally close to the original one. Because of that, the existence of conjugate points is a sign of (Lagrangian nonlinear) stability.

The same argument hold true for   the one dimensional central extension $\widehat{\mathcal{D}^s_{vol}}(\Sph^2)$ 
as presented in \cite{Tauchi-Yoneda-2}. 
%
%%%%%%%%%%%%%%%%%%%%%%%%%        Def MC for right invariant groups     %%%%%%%%%%%%%%%%%%%%%%%%%%%%%%%%%%%%%%%%%%%%%%%%%%%%
%

Let $(u,a)$ be a  solution of the Euler-Arnold equation. As a result, we state the following definition according to lemma B.6 from \cite{Tauchi-Yoneda-2}.
\begin{Def}
For $(u,a),(v,b)\in\hat\cg=\widehat{\mathcal{D}^s_q}(\Sph^3)$ the Misiolek criterion   is given by
\begin{eqnarray}\label{MC from Tauchi-Yoneda}
\nonumber\widehat{MC}\big( {(u,a),(v,b)} \big) &=&  - \ll  (w,d) ,(w,d) \gg_{\hat\cg}   -   \ll  [  (w,d) ,(v,b)],(u,a) \gg_{\hat\cg} \\
\nonumber&=&      \ll  \hat{ad}_{(u,a)}^*(w,d)     +  \hat{ad}_{(w,d)}^*(u,a)  ,  (v,b)      \gg_{\hat\cg} \\
&=&  \ll  \hat\nabla_{(u,a)}[(u,a),(v,b)] + \hat\nabla_{[(u,a),(v,b)]}(u,a) , (v,b) \gg_{\hat\cg}. 
\end{eqnarray}
where $(w,d)=([u,v],\Omega(u,v) )$.    
\end{Def}  
The same argument holds true for    $\widehat{\mathcal{D}^s_{vol}}(\Sph^2)$.

%%%%%%%%%%%%%%%%%%%%%%%%%%%%%%%%%%%%%%%%%%%%%%%%%%%%%%%%%%%%%%%%%%%%%%%%%%%%%%%%%%%%%%%%%%%%%%%%%%%%%%%%%%%%%%%%%%%%%%%%%%%%%%%%%%%%%%%%%%%%%%%%
%However
%
%\begin{eqnarray*}
% - \ll  (w,d) ,(w,d) \gg_{\hat\cg}  &=&  \ll \hat{ad}_{(u,a)}(v,b) , (w,d)\gg_{\hat\cg} \\
%
%&=&   \ll  (v,b) , \hat{ad}_{(u,a)}^*(w,d)\gg_{\hat\cg}
%\end{eqnarray*}
%
%and 
%
%\begin{eqnarray*}
%-   \ll  [  (w,d) ,(v,b)],(u,a) \gg_{\hat\cg}   &=&     \ll  \hat{ad}_{(w,d)} (v,b),(u,a) \gg_{\hat\cg} \\
%
%&=&     \ll   (v,b)   ,    \hat{ad}_{(w,d)}^*(u,a) \gg_{\hat\cg} 
%\end{eqnarray*}
%where $(w,d)=([u,v],\Omega(u,v) )$. As a result of the above argument and equality \eqref{eq symmetric covariant derivative},  the criterion %\ref{MC from Tauchi-Yoneda} is given by
%
%\begin{eqnarray*}
%\widehat{MC}\big( {(u,a),(v,b)} \big)   &=&      \ll  \hat{ad}_{(u,a)}^*(w,d)     +  \hat{ad}_{(w,d)}^*(u,a)  ,  (v,b)      \gg_{\hat\cg} \\
%
%&=&  \ll  \hat\nabla_{(u,a)}[(u,a),(v,b)] + \hat\nabla_{[(u,a),(v,b)]}(u,a) , (v,b) \gg_{\hat\cg} 
%\end{eqnarray*}
%

%
%%%%%%%%%%%%%%%%%%%%%%%%%    end     Rem MC for right invariant groups       %%%%%%%%%%%%%%%%%%%%%%%%%%%%%%%%%%%%%%%%%%%%%%%%%%%%
%
%
The next lemma represents the Misiolek criterion according to the stream functions, and its second part will play a crucial role in subsequent sections of the paper.
%
%%%%%%%%%%%%%%%%%%%%%%%%%%%%%%%%%%%%%%%%%%        lemma MC  with stream function  %%%%%%%%%%%%%%%%%%%%%%%%%%%%%%%%%
%
%
\begin{Lem}\label{lemma MC on Dq}
For   $u=S_\theta f$ and $v=S_\theta g$ we have\\
\textbf{(a)}
\begin{eqnarray}\label{eqn Misiolek Curvature on Dsq}
\nonumber\widehat{MC}\big( {(u,a),(v,b)} \big) &=&  \ll  \hat\nabla_{(u,a)}[(u,a),(v,b)] + \hat\nabla_{[(u,a),(v,b)]}(u,a) , (v,b) \gg_{\hat\cg}\\
&=& - \langle \Delta_\theta \{f,g\}  ,  \{f,g\}  \rangle   -  \langle  \{g,\Delta_\theta f\}  ,  \{f,g\}  \rangle  \\
\nonumber && -   \langle  \{\mu,f\}  ,  g  \rangle^2  -  a  \langle \{\mu, \{f,g\}\}  ,  g  \rangle.
\end{eqnarray}
\textbf{(b)} If $E_1 f=E_1g=0$ then,
\begin{eqnarray}\label{eqn Misiolek Curvature on S2}
\widehat{MC}\big( {(u,a),(v,b)} \big)  &=&  \langle \Delta \{f,g\}  ,  \{f,g\}  \rangle   -  \langle  \{\Delta f,g\}  ,  \{f,g\}  \rangle  \\
\nonumber&& -   \langle  \{\mu,f\}  ,  g  \rangle^2  -  a  \langle \{\mu, \{f,g\}\}  ,  g  \rangle.
\end{eqnarray}
\end{Lem}
%
%
%%%%%%%%%%%%%%%%%%%%%%%%%%%%%%%%%%%%%%%%%%%     Proof    %%%%%%%%%%%%%%%%%%%%%%%%%%%%%%%%%%%%%%%%%%%%%%%%
%
%
\begin{proof} \textbf{(a)}.
For $(u,a),(v,b)\in\hat{\cg}$ we have
\begin{eqnarray*}
\Omega(u,v)=\ll Tu,v\gg_\cg    =  \ll S_\theta\Delta_\theta^{-1}\{\mu,f\}  ,  S_\theta g\gg_\cg  =  \langle\{\mu,f\}  ,g \rangle
\end{eqnarray*}
which implies that $[(u,a),(v,b)]=(S_\theta\{f,g\}  , \langle  \{\mu,f\},g  \rangle ):=(w,d)$. Now, using equations (\ref{eq ad* on quantomorphism}) and (\ref{eq symmetric covariant derivative}) we get
\begin{eqnarray*}
\widehat{MC}\big( {(u,a),(v,b)} \big)  &=&  \ll  ( ad^*_uw  -  dTu  + ad^*_wu  - aTw  ,  0) , (v,b) \gg_{\hat\cg}\\
&=& \ll   ad^*_uw  + ad^*_wu   ,  v  \gg_{\cg}    -    \ll  dTu  + aTw   ,   v \gg_{\cg}.
\end{eqnarray*}
We note that
\begin{eqnarray*}
\ll   ad^*_uw  + ad^*_wu   ,  v  \gg_{\cg}  &=&   \ll ad^*_uw , v \gg_{\cg}     +     \ll ad^*_wu , v \gg_{\cg}\\
&=&  \ll w , ad_uv \gg_{\cg}     +     \ll u , ad_wv \gg_{\cg}\\
&=& - \ll w , [u,v] \gg_{\cg}     -     \ll u ,  ad_vw\gg_{\cg}\\
&=&  -  \ll S_\theta\{f,g\} ,  S_\theta\{f,g\}  \gg_{\cg}    -   \ll ad^*_vu ,  w\gg_{\cg}\\
&=&  -  \langle S_\theta^*S_\theta\{f,g\} ,  \{f,g\}  \rangle    -   \ll S_\theta\Delta_\theta^{-1}\{g,\Delta_\theta f\} ,  S_\theta\{f,g\}  \gg_{\cg}\\
&=&  -  \langle \Delta_\theta\{f,g\} ,  \{f,g\}  \rangle    -   \langle \{g,\Delta_\theta f\} ,  \{f,g\}  \rangle.
\end{eqnarray*}
Moreover  using the definition of the operator $T$ we have
\begin{eqnarray*}
\ll  dTu  + aTw   ,   v \gg_{\cg}  &=&  d  \ll    S_\theta \Delta_\theta^{-1}\{\mu,f\}   ,   S_\theta g \gg_{\cg} + a \ll S_\theta\Delta_\theta^{-1}\{\mu,\{f,g\}\}   ,   S_\theta g \gg_{\cg}\\
&=&   d \langle  \{\mu,f\}  ,  g\rangle   +  a \langle  \{\mu,\{f,g\}\} , g  \rangle\\
&=&   \langle  \{\mu,f\}  ,  g\rangle^2   +  a \langle  \{\mu,\{f,g\}\} , g  \rangle\\
\end{eqnarray*}
which completes the proof of part \textbf{(a)}.

\textbf{(b)}.  If $E_1 f=E_1g=0$ then the contact Laplacian reduces to $\Delta_\theta=\alpha^2-\Delta$  and consequently we have
\begin{eqnarray*}
\widehat{MC}\big( {(u,a),(v,b)} \big)  &=&  - \langle (\alpha^2-\Delta) \{f,g\}  ,  \{f,g\}  \rangle   -  \langle  \{g,(\alpha^2-\Delta) f\}  ,  \{f,g\}  \rangle  \\
&& -   \langle  \{\mu,f\}  ,  g  \rangle^2  -  a  \langle \{\mu, \{f,g\}\}  ,  g  \rangle\\
&=&  \alpha^2 \big[ - \langle  \{f,g\}  ,  \{f,g\}  \rangle   -  \langle  \{g,f\}  ,  \{f,g\}  \rangle\big]  \\
&&  + \langle  \Delta\{f,g\}  ,  \{f,g\}  \rangle   +  \langle  \{g,\Delta f\}  ,  \{f,g\}  \rangle\\
&& -   \langle  \{\mu,f\}  ,  g  \rangle^2  -  a  \langle \{\mu, \{f,g\}\}  ,  g  \rangle\\
&=&  \langle  \Delta\{f,g\}  ,  \{f,g\}  \rangle   -  \langle  \{\Delta f,g\}  ,  \{f,g\}  \rangle\\
&& -   \langle  \{\mu,f\}  ,  g  \rangle^2  -  a  \langle \{\mu, \{f,g\}\}  ,  g  \rangle.
\end{eqnarray*}
\end{proof}
%
%%%%%%%%%%%%%%%%%%%%%%%%%%%%%%%%%%%%%%%%%%%%%%%%%%%%%%%%%%%%%%        end proof      %%%%%%%%%%%%%%%%%%%%%%%%%%%%%%%%%%%%%%%%%%%%%%
%%%%%%%%%%%%%%%%%%%%%%%%%%%%%%%%%%%%%%%%%%%%%%%%%%%%%%%%%%%%%%        Remark and Corollary      %%%%%%%%%%%%%%%%%%%%%%%%%%%%%%%%%%%
%
%\begin{Rem}
%
In \cite{Tauchi-Yoneda-2}, Tauchi and Yoneda termed the preceding criterion as the 'Misiolek curvature' and introduced \eqref{eqn Misiolek Curvature on S2} for vector fields using a different approach. Specifically, \eqref{eqn Misiolek Curvature on S2} corresponds to a stream function version of equation (18) in \cite{Tauchi-Yoneda-2}. Additionally, Benn in \cite{Benn2021} presented a version of \eqref{eqn Misiolek Curvature on S2} for $\cg$, which can be derived from  \eqref{eqn Misiolek Curvature on S2}  by neglecting the central extension part.
%
%
%%%%%%%%%%%%%%%%%%%%%%%%%%%%%%%%%%%%%%%%%%%%%%%%    remark   b in MC     %%%%%%%%%%%%%%%%%%%%%%%%%%%%%%%%%%%%%%%%%%%%%%%%
%
%
\begin{Rem}
Looking at \eqref{eqn Misiolek Curvature on S2}, we notice that the real number $b$ in $(v, b) \in \hat\cg$ doesn't play a role in the criterion. However, it's worth mentioning that this parameter can affect the instant of the appearance of the conjugate point along $(u, a)$.

More precisely, suppose that there exists a $\kappa>0$ such that  $\widehat{MC}\big( {(u,a),(v,b)} \big)> \kappa  ||(v,b)||^3 >0$.  Set
\[
t_0:=\frac{\pi}{  \sqrt {\kappa ||(v,b)||}    }   \quad \textrm{and} \quad 
\psi(t):=\sin\Big(    t\sqrt {\kappa ||(v,b)||}   \Big).
\]
 Then, for the index $I_0^{t_0}$ \eqref{Index} we have 
\begin{eqnarray*}
&& I_0^{t_0}((v,b),(v,b)) \\
&&= \int_0^{t_0}  \Big(  \dot\psi(t)^2 ||(v,b)||^2  -  \psi(t)^2\big(   \widehat{MC}\big( {(u,a),(v,b)} \big)  \Big)dt\\
&&<     \int_0^{t_0}  \Big(  \dot\psi(t)^2 ||(v,b)||^2  -  \psi(t)^2 \kappa||(v,b)||^3 \Big)dt\\
&&=    \kappa||(v,b)||^3  \int_0^{t_0}  \Big(  \cos^2\Big(    t\sqrt {\kappa ||(v,b)||}   \Big)   -   \sin^2\Big(    t\sqrt {\kappa ||(v,b)||}   \Big)  \Big)ds\\
&&=  0
\end{eqnarray*}
As a result,  for the Jacobi field $J(t)=\psi(t)(v,b)\hat\eta(t)$ there exist $0<t_c\leq t_0$ such that $J(0)=J(t_c)=0$. 

Clearly, changing the parameter $b$ in $(v, b) \in \hat\hat\cg$ can alter the time $t_0$ for a fixed $(u, a) \in \hat\cg$. Specifically, when $b$ has a larger absolute value, $t_0$ increases.
\end{Rem}

%The criterion (\ref{eqn Misiolek Curvature on S2}) is comparable with equation (18) in Proposition 3.2 of \cite{Tauchi-Yoneda-2}. Despite of the different approaches, %equation (\ref{eqn Misiolek Curvature on S2}) is the stream function version of Proposition 3.2 in \cite{Tauchi-Yoneda-2}.
%\end{Rem}
%
%
\begin{cor}\label{Cor MC of zonal flow} In the case that  $f$ is a zonal flow i.e. $f=f(\mu)$ then, $\{\mu,f\}=0$  and (\ref{eqn Misiolek Curvature on S2}) reduces to
\begin{equation}\label{eqn Misiolek Curvature on S2 reduced}
\widehat{MC}\big( {(u,a),(v,b)} \big)  =  \langle \Delta \{f,g\}  ,  \{f,g\}  \rangle   -  \langle  \{\Delta f,g\}  ,  \{f,g\}  \rangle
+  a  \langle \frac{\partial}{\partial\lambda} \{f,  g\}  ,  g  \rangle.
\end{equation}
\end{cor}
%
%
%
%%%%%%%%%%%%%%%%%%%%%%%%%%%%%%%%%%%%%%%%%%%%%%%%%%%%%     Spherical harmonics and structure constants     %%%%%%%%%%%%%%%%%%%%%%%%%%%%%%%%%%%%%%%%%%%%%%%%%
%%%%%%%%%%%%%%%%%%%%%%%%%%%%%%%%%%%%%%%%%%%%%%%%%%%%%     Spherical harmonics and structure constants     %%%%%%%%%%%%%%%%%%%%%%%%%%%%%%%%%%%%%%%%%%%%%%%%%
%%%%%%%%%%%%%%%%%%%%%%%%%%%%%%%%%%%%%%%%%%%%%%%%%%%%%     Spherical harmonics and structure constants     %%%%%%%%%%%%%%%%%%%%%%%%%%%%%%%%%%%%%%%%%%%%%%%%%%
%
%
%
\subsection{Spherical harmonics and structure constants }
In this section, we will introduce spherical harmonics and structure constants, following the notations established in \cite{Arak-Sav, Edmonds, Mess, Dowker, Mozheng}. 
Due to the significant role that Wigner $3j$-symbols (or simply $3j$-symbols)  play in the quantum theory of angular momentum, there exists a vast literature discussing their properties. We will recall the concept of  $3j$-symbols, investigate their special closed forms for our purposes, and examine their connection to Clebsch-Gordon coefficients as detailed in \cite{Mess}, \cite{Edmonds}, and \cite{Varsha} for further investigations.

Note that, when transitioning from the framework of a flat torus, as discussed in \cite{Arnold66, Misiolek99, Preston-Brig}, to a sphere, it becomes natural (and necessary) to work with structure constants and $3j$ symbols.

Arakelyan and Savvidy, in their work \cite{Arak-Sav}, offer an approach for computing structure constants. However, it's worth noting that the notation employed by Dowker in \cite{Dowker} is more efficient, and we will adopt that notation for our purposes.\\
The first reference known to the author that contains a formula for  structure constants is Jones \cite{Jones}. For further insight  and discussion on this topic,  see also Chapter 2 of \cite{Mozheng}.

Consider the complex spherical harmonic $Y_{lm}:\Sph^2\rightarrow \C$ where
\begin{equation*}
Y_{lm}(\lambda,\mu)=C^m_l P^{|m|}_l(\mu)e^{im\lambda}
\end{equation*}
and for the integers  $0\leq |m|\leq l$ the coefficient  $C^m_l=(-1)^{m}\sqrt{   \frac{2l+1}{4\pi}\frac{(l-|m|)!}{(l+|m| )!}  }$.
$\{Y_{lm}\}$ are eigenfunctions of  the Laplacian with $\L Y_{lm}=-l(l+1)Y_{lm}$. Following the notations of \cite{Mess} the associated Legendre polynomial is given by
\begin{equation*}
P^{|m|}_l(\mu)=\frac{  {(1-\mu^2)}^{\frac{m}{2}}   }{2^ll!}  \frac{d^{l+|m|}}{d\mu^{l+|m|}}  (\mu^2-1)^l
\end{equation*}
with $-1\leq \mu\leq 1$ and $0\leq\lambda\leq 2\pi$. Moreover the complex conjugation is given by
$$Y_{lm}^*=(-1)^mY_{l-m}$$
and the formulas $P^{|m|}_l(-\mu)=(-1)^{l-|m|}P^{|m|}_l(\mu)$, $P^{-|m|}_l(\mu)=(-1)^m P^{|m|}_l(\mu)$ hold true.
Now suppose that $e_{lm}=\sgrad Y_{lm}$  where $\sgrad$ represents the skew gradient.    We will adopt the notation $\nabla^\perp$ for $\sgrad$ on $\Sph^2$. The  family of vector fields $\{e_{lm}\}$ form a basis for $\cg=T_e\D^s_{vol}(\Sph^2)$ where $e_{lm}=\nabla^\perp{Y}_{lm}$ and $\nabla^\perp Y_{lm}=(-\frac{\partial}{\partial\mu}Y_{lm} , \frac{\partial}{\partial\lambda}Y_{lm})$.
In this case for $a=\alpha=0$ the $\cL^2$-metric \eqref{Inner prod hat S3} reduces to
\begin{equation}\label{L2 metric}
\ll    \nabla^\perp f,  \nabla^\perp g     \gg_{ \cg} : = -\int_{\Sph^2} (\Delta f) gdA  =  -\int_{\Sph^2}f\Delta gdA.
\end{equation}
on $\cg$ where $dA$ represents the surface element of $\Sph^2$. In local coordinates we have
\[
\ll    \nabla^\perp f, \nabla^\perp g    \gg_\cg : = -\int_{-1}^{1}\int_0^{2\pi} f\Delta g d\lambda d\mu.
\]
 Moreover the Lie bracket on $ \cg$ is given by
\[
[\nabla^\perp f  , \nabla^\perp g]_{ \cg}:=    \nabla^\perp \{f  ,  g\}
\]
Due to the fact that
\begin{equation}\label{inner pproduct}
\langle   Y_{l_1m_1}  ,   Y_{l_2m_2}    \rangle  = \int_{-1}^{1}\int_0^{2\pi} Y_{l_1m_1}  Y_{l_2m_2} d\lambda d\mu  =  (-1)^{m_1} \delta^{l_1}_{l_2}\delta^{m_1}_{-m_2}
\end{equation}
we have
\begin{eqnarray*}
\ll   e_{l_1m_1}   ,    e_{l_2m_2}     \gg_{ \cg}    & = &     \ll  \nabla^\perp  Y_{l_1m_1}  ,   \nabla^\perp Y_{l_2m_2}    \gg_{\cg}   \\
& = &  \langle    -\Delta  (Y_{l_1m_1})  ,   Y_{l_2m_2}    \rangle   \\
& = &  \langle     l_1(l_1+1)  Y_{l_1m_1}  ,   Y_{l_2m_2}    \rangle  \\
& = &     l_1(l_1+1)   (-1)^{m_1} \delta^{l_1}_{l_2}\delta^{m_1}_{-m_2}   .\\
\end{eqnarray*}
Note that when $e_{l_2m_2}$ is complex, or equivalently $m_2\neq 0$, we implicitly consider the inner product as below
\begin{eqnarray*}
\ll   e_{l_1m_1}   ,    e_{l_2m_2}     \gg_{ \cg}    =\ll   e_{l_1m_1}   ,    e_{l_2m_2}^*   \gg_{ \cg}.
\end{eqnarray*}
Let
\begin{equation}\label{structure constants}
\{Y_{l_1m_1},Y_{l_2m_2}\}:=G^{l_3m_3}_{l_1m_1l_2m_2} Y_{l_3m_3}
\end{equation}
or equivalently $[e_{l_1m_1},e_{l_2m_2}]:=G^{l_3m_3}_{l_1m_1l_2m_2} e_{l_3m_3} $ where we used the Einstein summation convention.
The real structure constants $g_{l_1 m_{1} l_{2} m_{2}}^{l_{3}m_{3}}$ are defined by
\begin{equation}\label{realstr constants}
G_{l 1 m_{1} l_2 m_{2}}^{l_{3} m_{3}}=-\mathrm{i}(-1)^{m_{3}} g_{l_1 m_{1} l_{2} m_{2}}^{l_{3}-m_{3}}
\end{equation}
and there are important properties for them \cite{Arak-Sav}. In particular
\begin{equation}\label{symmetry property for glm}
g_{l_1 m_{1} l_{2} m_{2}}^{l_{3} m_{3}}= g_{l_3 m_{3} l_{1} m_{1}}^{l_{2} m_{2}}=g_{l_2 m_{2} l_{3} m_{3}}^{l_{1} m_{1}},
\end{equation}
\begin{equation}\label{negative for glm}
g_{l_1 -m_{1} l_{2} -m_{2}}^{l_{3} -m_{3}}  =  - g_{l_1 m_{1} l_{2} m_{2}}^{l_{3} m_{3}},
\end{equation}
and
\begin{equation}\label{negative for changing lower positions glm}
g_{l_1 m_{1} l_{2} m_{2}}^{l_{3} m_{3}}  =  - g_{l_2 m_{2} l_{1} m_{1}}^{l_{3} m_{3}}
\end{equation}
Moreover the structure constant $g_{l_1 m_{1} l_{2} m_{2}}^{l_{3} m_{3}}$ vanishes if $l_1+l_2+l_3$ is an even number and $g_{l_1 m_{1} l_{2} m_{2}}^{l_{3} m_{3}}=0$ if $m_1+m_2+m_3 \neq 0$. For an explicit definition of $3j$-symbols using Racah formula see e.g. \cite{Mess}, page 1058, equation C21.\\
%
%
%
%%%%%%%%%%%%%%%%%%%%%%%%%%%%%%%%%%%%%    formula for 3j symbols    %%%%%%%%%%%%%%%%%%%%%%%%%%%%%%%%%%%%%%
%
%
We remind that the  Wigner $3j$-symbol
\begin{equation}\label{3j symbol}
\left(
\begin{array}{ccc}
l_1          &   l_2  &  l_3\\
m_1          &   m_2  &   m_3
\end{array}\right)
\end{equation}
is a real number which is zero if the  conditions  \textbf{ 1.} $m_1+m_2+m_3=0$,
\textbf{ 2.} $|l_1-l_2| \leq l_3 \leq l_1+l_2$  are \textbf{not} met. (For more details see   \cite{Mess}, appendix C, part I or \cite{Varsha} chapter 8).
According to \cite{Varsha}, chapter 8, section 2, the relation between $3j$-symbols and Clebsch- Gordon coefficients $C^{l_3m_3}_{l_1m_1l_2m_2}$ is given by
\begin{eqnarray}\label{3j and Clebsch Gordon}
\left(
\begin{array}{ccc}
l_1    &   l_2  &  l_3\\
m_1          &   m_2  &   m_3
\end{array}\right)  = (-1)^{l_3+m_3}\frac{1}{\sqrt{2l_3+1}}C^{l_3m_3}_{l_1-m_1l_2-m_2}
\end{eqnarray}
We will use the following useful closed forms for structure constants from \cite{Edmonds}. After modifying formula (3.7.11) from page 48 \cite{Edmonds} we get
\begin{eqnarray}\label{3j Edmond j1 m1}
&&  \left(
\begin{array}{ccc}
l_1    &   m  &  l_3\\
m_1          &  - m  &   -m_1+m
\end{array}\right)  =  \left(-1\right)^{l_1 -m_1} \Big(   \frac{\left(2 m \right)!}{  \left(l_1+l_3+m  +1\right)!   }\\
&&  \times  \frac{ \left(l_1+l_3 -m  \right)! \left(l_3- m_1 + m  \right)! \left({l_1} +{m_1} \right)!}{    \left(l_1-l_3 +m  \right)! \left(-l_1 + l_3 + m  \right)! \left(l_3 + m_1 -m    \right)! \left( l_1 -m_1 \right)!}  \Big)^{\frac{1}{2}}\nonumber
\end{eqnarray}
Moreover formula 3.7.15, page 49 of \cite{Edmonds} implies that if $J=l_1+l_2+l_3+1$ is even then
%
%This was verified in the maple program okt 11.2023 equations 64-74 output of function fda
%
\begin{eqnarray}\label{3j Edmond j1 1 j2 -1}
&&  \left(
\begin{array}{ccc}
l_1    &   l_2  &  l_3\\
1      &  - 1  &   0
\end{array}\right)  = (-1)^{\frac{J}{2}}  \, (\frac{J}{2}  )!\\
&&\frac{   \Big(   \frac{ (J +1 ) (J -2  l_3 ) (J -2 l_1 ) (J -2 l_2 -1 )}{l_1 (l_1 + 1 ) {l_2} \left(l_2+ 1 \right)}            \times             \frac{\left(J -2 {l_3} \right)! \left(J -2 {l_1} \right)! \left(J -2 {l_2} -2\right)!}{\left(J +1\right)!}  \Big)^{\frac{1}{2}}    }{    2 \left(\frac{J}{2}-{l_3} \right)! \left(\frac{J}{2}-{l_1} \right)! \left(\frac{J}{2}-{l_2} -1\right)!} \nonumber
\end{eqnarray}
and in the case that $J=l_1+l_2+l_3+1$ is odd, then the $3j$-symbols vanishes.\\
Note that  \eqref{3j Edmond j1 m1} and \eqref{3j Edmond j1 1 j2 -1} are nonzero if the triangle inequality are satisfied. Finally the following useful recursive relations 
\begin{equation}\label{3j Pres j1 1}
\left(
\begin{array}{ccc}
l_1    &   l_2  &  l_3\\
1      &  1     &  -2
\end{array}\right)    = \left(-1\right)^{l_1+l_2+l_3}   \frac{\left(l_1-l_2 \right)(l_1+l-2+1)}{  l_1(l_1+1)\big( l_3(l_3+1)-2 \big)   }
\left(
\begin{array}{ccc}
l_1    &   l_3  &  l_2\\
1      &  -1     &  0
\end{array}\right) .
\end{equation}
is from \cite{Pain} equation (57). 
Instead of the previous recursive relation, one may utilize equation (8) on page 253 of \cite{Varsha}, along with the interchange formula \eqref{3j and Clebsch Gordon}.

%
%
%
%
%%%%%%%%%%%%%%%%%%%%%%%%%%%%%%%%%%%%%%%%%%%%%   Dowker structure constants    %%%%%%%%%%%%%%%%%%%%%%%%
%
%
According to Dowker \cite{Dowker} the structure constants are given by
\begin{eqnarray}\label{Dowker's formula for str cts}
&&   g^{l_3m_3}_{l_1m_1l_2m_2}  = \frac{-1}{  \sqrt{4\pi}  }  L_{123}
\left(
\begin{array}{ccc}
l_1    &   l_2  &  l_3\\
m_1          &   m_2  &   m_3
\end{array}\right)
\left(
\begin{array}{ccc}
l_1    &   l_2  &  l_3\\
1      &   -1   &   0
\end{array}\right)
\end{eqnarray}
where
\begin{equation}\label{eq L123}
L_{123}=[ (2l_1 +1) (2l_2+1) (2l_3+1) l_1(l_1+1) l_2(l_2+1)]^{\frac{1}{2}}
\end{equation}
%
%
%%%%%%%%%%%%%%%%%%%%%%%%%%%%%%%%%%%%%%%%%%%%%%    begin Corollary    %%%%%%%%%%%%%%%%%%%%%%%%
%
%
\begin{cor}\label{Cor bounds for str cts}
If the parameter $l_3$ is \textbf{not}  between $|l_1-l_2|+1 \leq    l_3  \leq 1_1+l_2-1$ then $g^{l_3m_3}_{l_1m_1l_2m_2}$ vanishes.
\end{cor}
\begin{proof}
If
$$   |l_1- l_2|    \leq      l_3     \leq    l_1    +    l_2              $$
is not satisfied, then the  $3j$-symbol \eqref{3j symbol} vanishes. Equivalently the  right hand side of  (\ref{Dowker's formula for str cts}) is zero  which implies that  $g^{l_3m_3}_{l_1m_1l_2m_2}=0$. On the other hand, if $l_3=l_1+l_2$, then $l_1+l_2+l_3=2l_3$ is even and the structure constant is zero too. The same holds true when $l_2=|l_1-l_2|$. As a consequence, if  $|l_1 - l_2| +1 \leq l_3 \leq l_1 +l_2-1$ is not satisfied then $g^{l_3m_3}_{l_1m_1l_2m_2}=0$.
\end{proof}
As a result of the previous corollary we  can write
\begin{eqnarray}\label{formula str cts reduced sum}
\{Y_{l_1m_1},Y_{l_2m_2}\}   &=&    \sum_{l_3 =|l_1-l_2| +1}^{l_1+l_2 -1} \sum_{m_3=-l_3}^{l_3}G^{l_3m_3}_{l_1m_1l_2m_2} Y_{l_3m_3}\\
&=&  -i  \sum_{l_3 =  |l_1-l_2| +1 }^{l_1+l_2 -1} \sum_{m_3=-l_3}^{l_3}   (-1)^{m_3}  g^{l_3-m_3}_{l_1m_1l_2m_2} Y_{l_3m_3}  \nonumber\\
&=&  -i (-1)^{m_1+m_2}   \sum_{l_3 = |l_1-l_2| +1}^{l_1+l_2-1}      g^{l_3  -(m_1+m_2)}_{l_1m_1l_2m_2} Y_{l_3 m_1+m_2}.  \nonumber
\end{eqnarray}
%
%
%
%%%%%%%%%%%%%%%%%%%%%%%%%%%%%%%%%%%%%%%%%%%%%%%%%%%%%%%       Conjugate points and Misiolek criterion  %%%%%%%%%%%%%%%%%%%%%%%%%%%%%%%%%%%%%%%%%%%%%%%%%%%%%%%%%%
%%%%%%%%%%%%%%%%%%%%%%%%%%%%%%%%%%%%%%%%%%%%%%%%%%%%%%%       Conjugate points and Misiolek criterion  %%%%%%%%%%%%%%%%%%%%%%%%%%%%%%%%%%%%%%%%%%%%%%%%%%%%%%%%%%
%%%%%%%%%%%%%%%%%%%%%%%%%%%%%%%%%%%%%%%%%%%%%%%%%%%%%%%       Conjugate points and Misiolek criterion  %%%%%%%%%%%%%%%%%%%%%%%%%%%%%%%%%%%%%%%%%%%%%%%%%%%%%%%%%%
%
%
%
\section{  Conjugate points  along spherical harmonics on $\D^s_{vol}(\Sph^2)$}\label{Sect. CP S.H}
In this section, we will restate the Misiolek criterion from Lemma \ref{lemma MC on Dq}, using structure constants, specifically in the context where the Coriolis force is absent.
When there is no rotation, we can simply set $\phi=0$. As a result, for the velocity fields given by $u = \nabla^\perp f$ and $v = \nabla^\perp g \in \cg$, lemma \ref{lemma MC on Dq} part "b" implies that
\begin{eqnarray}\label{eqn Misiolek Curvature on Dsq}
\nonumber{MC}(u , v )  = \langle \Delta \{f,g\}  ,  \{f,g\}  \rangle   -  \langle  \{\Delta f,g\}  ,  \{f,g\}  \rangle.
\end{eqnarray}
%
%
%
%%%%%%%%%%%%%%%%%%%%%%%%%%%%%%%%%%%%%%%%%%%%%%%%%%%%%%    Prop Mis Curv e    %%%%%%%%%%%%%%%%%%%%%%%%%%%%%%%%%%%
%
%
\begin{prop}\label{Prop Mis Curv e}
For $f=Y_{l_1m_1}$, $g=Y_{l_2m_2}$  we have
\begin{eqnarray}\label{MC with str cts}
{MC}(  e_{l_1m_1} , e_{l_2m_2} )  &=& \sum_{l_3=  |l_1-l_2| +1  }^{l_1+l_2-1}   \sum_{m_3=-l_3}^{l_3}    \Big( g^{l_3-m_3}_{l_1m_1l_2m_2} \Big)^2 \Big( l_1(l_1+1)  -   l_3(l_3+1)  \Big)\nonumber \\
&=&  \sum_{l_3=  |l_1-l_2| +1  }^{l_1+l_2-1}     \Big( g^{l_3-(m_1+m_2)}_{l_1m_1l_2m_2} \Big)^2 \Big( l_1(l_1+1)  -   l_3(l_3+1)  \Big)
\end{eqnarray}
\end{prop}
\begin{proof}
Since
\begin{equation*}
\{f,g\}=G^{l_3m_3}_{l_1m_1l_2m_2} Y_{l_3m_3}
\end{equation*}
we have
\begin{eqnarray*}
\langle \Delta \{f,g\}  ,  \{f,g\}  \rangle   &=&    -l_3(l_3+1)G^{l_3m_3}_{l_1m_1l_2m_2} ({G^{l_4m_4}_{l_1m_1l_2m_2}})^*   \langle Y_{l_3m_3}  ,  ({Y_{l_4m_4}})^*  \rangle\\
&=&  \sum_{l_3,m_3}  -l_3(l_3+1)G^{l_3m_3}_{l_1m_1l_2m_2} (G^{l_3m_3}_{l_1m_1l_2m_2})^*
\end{eqnarray*}
and
\begin{eqnarray*}
\langle  \{\Delta f,g\}  ,  \{f,g\}  \rangle  &=&  -l_1(l_1+1)G^{l_3m_3}_{l_1m_1l_2m_2}  (G^{l_4m_4}_{l_1m_1l_2m_2} )^*  \langle Y_{l_3m_3}  ,  (Y_{l_4m_4})^*  \rangle\\
&=& \sum_{l_3,m_3}  - l_1(l_1+1)G^{l_3m_3}_{l_1m_1l_2m_2} (G^{l_3m_3}_{l_1m_1l_2m_2})^*
\end{eqnarray*}
Using the real structure constants   we get
\begin{eqnarray*}
&&  {MC}( e_{l_1m_1} , e_{l_2m_2} )  =    \sum_{l-3,m_3} G^{l_3m_3}_{l_1m_1l_2m_2} (G^{l_3m_3}_{l_1m_1l_2m_2})^* \Big(  l_1(l_1+1)  -l_3(l_3+1)  \Big)\\
&=&  \sum_{l_3,m_3}   |G^{l_3m_3}_{l_1m_1l_2m_2} |^2 \Big(  l_1(l_1+1)  -l_3(l_3+1)  \Big)\\
&=& \sum_{l_3,m_3   } (-i)(-i)^* (g^{l_3-m_3}_{l_1m_1l_2m_2})^2   \Big(  l_1(l_1+1)  -l_3(l_3+1)  \Big) \\
&=& \sum_{l_3,m_3   }  (g^{l_3-m_3}_{l_1m_1l_2m_2})^2   \Big(  l_1(l_1+1)  -l_3(l_3+1)  \Big)
\end{eqnarray*}
which completes the proof.
\end{proof}
%
%%%%%%%%%%%%%%%%%%%%%%%%%%%%%%%%%%%%%%%%%%%%%%%%%%%%   end proof  %%%%%%%%%%%%%%%%%%%%%%%%%%%%%%%%%%%%%%%%%%%%%%%%%%
%
%
%
%%%%%%%%%%%%%%%%%%%%%%%%%%%%%%%%%%%%%%%%%%%%%%      Corollary  l_1=1  %%%%%%%%%%%%%%%%%%%%%%%%%%%%%%%%%%%%%%%%%%%%%%
%
%

\begin{cor}\label{Cor-MC-l1=1}
\textbf{i.} If $l_2=1$, then for any $l_1$, we have $l_1+l_2-1=|l_1-l_2|+1=l_1$, and consequently
\begin{eqnarray*}
{MC}(  e_{l_1m_1} , e_{1~m_2} )  &=& \sum_{l_3=  l_1  }^{l_1}       \Big( g^{l_3-m_3}_{l_1m_1~1~m_2} \Big)^2 \Big( l_1(l_1+1)  -   l_3(l_3+1)  \Big)\\
&=& \Big( g^{l_1-(m_1+m_2)}_{l_1~m_1~1~m_2} \Big)^2 \Big( l_1(l_1+1)  -   l_1(l_1+1)  \Big)=0
\end{eqnarray*}
the criterion \eqref{MC with str cts} vanishes. \\
\textbf{ii.} For $l_1=1$ and any $1\leq l_2$ we have
\begin{eqnarray*}
{MC}(  e_{1~m_1} , e_{l_2m_2} )  &=& \sum_{l_3=  l_2-1 +1  }^{l_2+1-1}       \Big( g^{l_3-m_3}_{1~m_1l_2m_2} \Big)^2 \Big( 2  -   l_3(l_3+1)  \Big)\\
&=& \Big( g^{l_2-(m_1+m_2)}_{1~m_1l_2m_2} \Big)^2 \Big( 2  -   l_2(l_2+1)  \Big) \leq 0
\end{eqnarray*}
Note that, in the special case $l_1=1$, the real and imaginary parts of the vector field $e_{1~1}$ are killing fields generated by (rigid) rotation and we already know that Misiolek criterion can not detect conjugate points. However, conjugate points along this field exist (see e.g. \cite{Preston-2023}).

%fields  As a result, when $l_1=1$, one may need to consider using nontrivial linear combinations of eigenfunctions in %the second component of the Misiolek criterion, as discussed in \cite{Brigant-Preston}.\\
%
%
\textbf{iii.}
For any two vector fields $e_{l_1m_1},e_{l_2m_2}$ we have
\begin{eqnarray*}
{MC}(  e_{l_1m_1} , e_{l_2m_2} )  &=& \sum_{l_3=  |l_1-l_2| +1  }^{l_1+l_2-1}       \Big( g^{l_3-(m_1+m_2)}_{l_1m_1l_2m_2} \Big)^2 \Big( l_1(l_1+1)  -   l_3(l_3+1)  \Big)\\
&=&   \sum_{l_3=  |l_1-l_2| +1  }^{l_1+l_2-1}       \Big( -g^{l_3-(-m_1-m_2)}_{l_1-m_1l_2-m_2} \Big)^2 \Big( l_1(l_1+1)  -   l_3(l_3+1)  \Big)\\
&=&  {MC}(  e_{l_1-m_1} , e_{l_2 -m_2} )
\end{eqnarray*}
\end{cor}

The Misiolek criterion is not generally linear. However, by utilizing the properties of the structure constants, we can state the following proposition.
%
%
%
%%%%%%%%%%%%%%%%%%%%%%%%%%%%%%%%%%%%%%%%%%%%%%%%%%    Prop Mis Curv linearity    %%%%%%%%%%%%%%%%%%%%%%%%%%%%%%%%%%%
%
%
%
\begin{prop}\label{Prop Mis Curv linearity}
Let $l_1,l_1^\prime, l_2, l_2^\prime\in\mathbb{N}$ and$x\in \mathbb{C}$  be arbitrary.\\
\textbf{i.}  For $m_2\neq m_2^\prime$   we have
\begin{equation*}
{MC}(  e_{l_1m_1} , e_{l_2m_2} + x e_{l_2^\prime m_2^\prime})  = {MC}(  e_{l_1m_1} , e_{l_2m_2} ) + |x|^2{MC}(  e_{l_1m_1} ,  e_{l_2^\prime m_2^\prime})
\end{equation*}
\textbf{ii.}  For $m_1\neq m_1^\prime$  we have
\begin{equation*}
{MC}(  e_{l_1m_1} + x e_{l_1^\prime m_1^\prime}, e_{l_2m_2} )  = {MC}(  e_{l_1m_1} , e_{l_2m_2} ) + |x|^2{MC}(  e_{l_1^\prime m_1^\prime} ,  e_{l_2 m_2})
\end{equation*}
\end{prop}
\begin{proof}
Following the method of proposition \ref{Prop Mis Curv e}, we have
\begin{eqnarray*}
&&  {MC}(  e_{l_1m_1} , e_{l_2m_2} + x e_{l_2^\prime m_2^\prime})\\
&=&   \sum_{l_3,m_3}    \Big| g^{l_3-m_3}_{l_1m_1l_2m_2} +x g^{l_3-m_3}_{l_1m_1l_2^\prime m_2^\prime} \Big|^2 \Big( l_1(l_1+1)  -   l_3(l_3+1)  \Big).
\end{eqnarray*}
For $m_2\neq m_2^\prime$ the term containing $g^{l_3-m_3}_{l_1m_1l_2m_2} g^{l_3-m_3}_{l_1m_1l_2^\prime m_2^\prime}$ vanishes since
\begin{eqnarray*}
\sum_{m_3}    g^{l_3-m_3}_{l_1m_1l_2m_2} g^{l_3-m_3}_{l_1m_1l_2^\prime m_2^\prime}
&=&
g^{l_3-(m_1+m_2)}_{l_1m_1l_2m_2} g^{l_3-(m_1+m_2)}_{l_1m_1l_2^\prime m_2^\prime}
+
g^{l_3-(m_1+m_2^\prime)}_{l_1m_1l_2m_2} g^{l_3-(m_1+m_2^\prime)}_{l_1m_1l_2^\prime m_2^\prime}\\
&=&0
\end{eqnarray*}
which completes the proof of part \textbf{i}. Part \textbf{ii} can be proved with the same method.
\end{proof}
%
%%%%%%%%%%%%%%%%%%%%%%%%%%%%%%%%%%%%%%%%%%%%%%%%%%%%   end proof  %%%%%%%%%%%%%%%%%%%%%%%%%%%%%%%%%%%%%%%%%%%%%%%%%%%%%%
%
%
%
%
%
%
The next theorem ensures us that along any vector field $e_{l_1m_1}$ with $1<|m_1|\leq l_1$ conjugate points exist and as a result the corresponding geodesics are not global length minimizers.
%
%We know that positivity of ${MC}( e_{l_1m_1} , e_{l_2m_2} )$ results existence of conjugate points.
%
%%%%%%%%%%%%%%%%%%%%%%%%%%%%%%%%%%%%%%%%%%%%%%%%%%%%%%    Theorem   %%%%%%%%%%%%%%%%%%%%%%%%%%%%%%%%%%%%%%%%%%%%%%%%%%%%%%%%%%%%%%
%
\begin{The}\label{Theorem MC S.H.}
Let $l_1$ be a natural number and $m_1$ be an integer with $1<  |m_1| \leq l_1$. Then,\\
\textbf{i.} for any $2\leq m\leq m_1$ we have $MC(e_{l_1 m_1} ,e_{m-m})>0$. \\
\textbf{ii.} For any $3\leq l_1$ we have $MC(e_{l_1 1} ,e_{l_2~1})>0$  for any $2\leq l_2<l_1$. 
 \\
\end{The}
\begin{proof}
%
%%%%%%%%%%%%%%%%%%%%%%%%%%%%%%%%%%%%%%%%%          Proof              %%%%%%%%%%%%%%%%%%%%%%%%%%%%%%%%%%%%%%%%%%%%
%
%
%
For the proof or part \textbf{i.}, we consider two cases.\\
\textbf{Case 1.} Let $m$ be even. Then for $l_1-m+1\leq l_3 \leq l_1+m-1$  in $g^{l_3~-m_1+m}_{l_1m_1 m~-m}$  the summation $L:=l_1+m+l_3$ is $L=l_1+m+l_1\pm j=2l_1+m\pm j$ for some $j\in\mathbb{N}\cup\{0\}$. As a result
$g^{l_3~-m_1+m}_{l_1m_1 m~-m}$ is nonzero if $j=\pm (2k+1)$ where
\begin{eqnarray*}
l_1-m+1 \leq l_1-(2k+1)
\end{eqnarray*}
or equivalently $0\leq k\leq\frac{m-2}{2}=[\frac{m-1}{2}]$. The summation appeared in \eqref{MC with str cts} is symmetric with respect to the new variable $k$ that is
\begin{eqnarray*}
&&  {MC}(  e_{l_1m_1} , e_{m~-m} )  \\
&=& \sum_{k=0 }^{\frac{m-2}{2}}       \Big( g^{l_1-(2k+1)~-m_1+m}_{l_1m_1 m~-m} \Big)^2 \Big( l_1(l_1+1)  -   (l_1-2k-1)(l_1-2k)  \Big)\\
&& -  \Big( g^{l_1+(2k+1)~-m_1+m}_{l_1m_1 m~-m} \Big)^2 \Big( -l_1(l_1+1)  +   (l_1+2k+1)(l_1+2k+2)  \Big)\\
\end{eqnarray*}
We will prove that for any $k$
\begin{eqnarray*}
&&\Big( g^{l_1-(2k+1)~-m_1+m}_{l_1m_1m~-m} \Big)^2 \Big( l_1(l_1+1)  -   (l_1-2k-1)(l_1-2k)  \Big) >\\
&&\Big( g^{l_1+(2k+1)~-m_1+m}_{l_1m_1m~-m} \Big)^2 \Big( l_1(l_1+1)  -   (l_1+2k+1)(l_1+2k+2)  \Big)
\end{eqnarray*}
or equivalently for any $0\leq k+1\leq \frac{m-2}{2}$
\begin{equation*}
h(l_1,m_1,m,k):=\frac{  \Big( g^{l_1-(2k+1)~-m_1+m}_{l_1m_1m~-m} \Big)^2 \Big( l_1(l_1+1)  -   (l_1-2k-1)(l_1-2k)  \Big)  }
{\Big( g^{l_1+(2k+1)~-m_1+m}_{l_1m_1m~-m} \Big)^2 \Big( -l_1(l_1+1)  +   (l_1+2k+1)(l_1+2k+2)  \Big)}>1
\end{equation*}
In fact we show that
\begin{eqnarray*}
1< h(l_1,m_1,m,0) &<& \dots< h(l_1,m_1,m,k) <  h(l_1,m_1,m,k+1) \\
&<&  \dots <  h(l_1,m_1,m,\frac{m-2}{2}).
\end{eqnarray*}
which implies that ${MC}(  e_{l_1m_1} , e_{m~-m} )>0$.\\
Using equations \eqref{3j Edmond j1 m1}, \eqref{3j Edmond j1 1 j2 -1}, \eqref{Dowker's formula for str cts} and the fact that $2\leq m$ we observe that
\begin{eqnarray*}
&& h(l_1,m_1,m,0):=\frac{  \Big( g^{l_1-1~-m_1+m}_{l_1m_1m~-m} \Big)^2 \Big( l_1(l_1+1)  -   (l_1-1)l_1  \Big)  }
{\Big( g^{l_1+1~-m_1+m}_{l_1m_1m~-m} \Big)^2 \Big( -l_1(l_1+1)  +   (l_1+1)(l_1+2)  \Big)}\\
&=&\frac{  2l_1-1}{2l_1+3} \frac{  \left(
\begin{array}{ccc}
l_1          &   m  &  l_1-1\\
m_1          &   -m    &    -m_1+m
\end{array}\right)^2
\left(
\begin{array}{ccc}
l_1          &   m  &  l_1-1\\
1          &   -1    &    0
\end{array}\right)^2       }{       \left(
\begin{array}{ccc}
l_1          &   m  &  l_1+1\\
m_1          &   -m    &    -m_1+m
\end{array}\right)^2   \left(
\begin{array}{ccc}
l_1          &   m  &  l_1+1\\
1          &   -1    &    0
\end{array}\right)^2}       \frac{  l_1}{l_1+1}   \\
&=&\frac{   \left( 2 l_1  + m +1\right)^{2} \left(l_1 + m_1 -m \right) \left(l_1 + m_1 - m  +1 \right)   }{    \left(2 l_1 - m  +1\right)^{2} \left( l_1 -m_1 +m  \right) \left(l_1 -m_1 +m +1 \right) }\frac{ l_1 \left(2 l_1 -1\right)  }{ \left( l_1 +1 \right)\left(2 l_1 +3\right)  }\\
&\geq&
\frac{   \left( 2 l_1  + m +1\right)   }{    \left(2 l_1 - m  +1\right)   }   \frac{ l_1    }{ \left( l_1 +1 \right)  }>1\\
\end{eqnarray*}
Moreover for any $0\leq k+1\leq \frac{m-2}{2}$ we have 
\begin{eqnarray*}
&&  \frac{h(l_1,m_1,m,k+1)}{h(l_1,m_1,m,k)}\\
&=&  \frac{  \Big( g^{l_1-2k-3~-m_1+m}_{l_1m_1m~-m} \Big)^2 \Big( g^{l_1+2k+1~-m_1+m}_{l_1m_1m~-m} \Big)^2 }
{\Big( g^{l_1+2k+3~-m_1+m}_{l_1m_1m~-m} \Big)^2  \Big( g^{l_1-2k-1~-m_1+m}_{l_1m_1m~-m} \Big)^2  }
\frac{    (l_1-k-1)(l_1+k+1)   }{  (l_1+k+2)(l_1-k)  }\\
&=&  \frac{    (2l_1-4k-5)(2l_1+4k+3)   }{  (2l_1+4k+7)(2l_1-4k-1)  }\\
&&  \times  \frac{  \left(  \begin{array}{ccc}  l_1          &   m  &  l_1-2k-3\\  m_1          &   -m    &    -m_1+m  \end{array}\right)^2
\left(\begin{array}{ccc}  l_1          &   m  &  l_1+2k+1\\  m_1          &   -m    &    -m_1+m \end{array}\right)^2   }
{
\left(\begin{array}{ccc}  l_1          &   m  &  l_1-2k-1\\  m_1          &   -m    &    -m_1+m \end{array}\right)^2
\left(  \begin{array}{ccc}  l_1          &   m  &  l_1+2k+3\\  m_1          &   -m    &    -m_1+m  \end{array}\right)^2}\\
%
%
%&&====================\\
%
&&  \times   \frac{
\left(  \begin{array}{ccc}  l_1          &   m  &  l_1-2k-3\\  1          &   -1    &    0  \end{array}\right)^2
\left(\begin{array}{ccc}  l_1          &   m  &  l_1+2k+1\\  1          &   -1    &    0 \end{array}\right)^2   }
{
\left(  \begin{array}{ccc}  l_1          &   m  &  l_1-2k-1\\  1          &   -1    &    0  \end{array}\right)^2
\left(\begin{array}{ccc}  l_1          &   m  &  l_1+2k+3\\  1          &   -1    &    0 \end{array}\right)^2}\\
&&  \times\frac{    (l_1-k-1)(l_1+k+1)   }{  (l_1+k+2)(l_1-k)  }\\
\end{eqnarray*}
Now, using \eqref{3j Edmond j1 m1} and \eqref{3j Edmond j1 1 j2 -1} and the fact that $2k+4\leq m$ we obtain
\begin{eqnarray*}
&&  \frac{h(l_1,m_1,m,k+1)}{h(l_1,m_1,m,k)}\\
%
%%%%%%%%%%%%%%%%%%%%%%%%%%%%%%%%%%%%%%%%%%%%%%%%%%%%%%%%%%%%%%%%%%%%%%%%%%%
%
&=&    \frac{ ( 2 l_1 +m -2 k -1    )^{2}   ( l_1 -m +m_1 - 2 k - 2  )   ( l_1 - m + m_1 -2 k -1   )       }{    ( 2 l_1 - m - 2 k  -1 )^{2}   ( l_1 + m - m_1  -2 k - 2  )  ( l_1 + m   - m_1 - 2 k - 1   )    }\\
&&\times  \frac{  (2 l_1 +m +2 k +3  )^{2} (l_1 -m  +m_1 +2 k  +2 )  ( l_1 - m  + m_1 + 2 k +3   )  }{   (  2 l_1 -m +2 k +3 )^{2} ( l_1 + m  - m_1   +2 k  +2 )
(l_1 + m  -m_1 +2 k +3  ) }   \\
&& \frac{    (2l_1-4k-5)(2l_1+4k+3)   }{  (2l_1+4k+7)(2l_1-4k-1)  }
\times\frac{    (l_1-k-1)(l_1+k+1)   }{  (l_1+k+2)(l_1-k)  }\\
%
%%%%%%%%%%%%%%%%%%%%%%%%%%%%%%%%%%%%%%%%%%%%%%%%%%%%%%%%%%%%%%%%%%%%%%%%%%%%%
%
%
&  >  &    \frac{ ( 2 l_1 +m -2 k -1    )        }{    ( 2 l_1 - m - 2 k  -1 )     }      \frac{  (2 l_1 +m +2 k +3  )   }{   (  2 l_1 -m +2 k +3 ) }   \frac{    (2l_1-4k-5)   }{  (2l_1-4k-1)  }\\
&&
\times\frac{    (l_1-k-1)(l_1+k+1)   }{  (l_1+k+2)(l_1-k)  }\\
%
%
%%%%%%%%%%%%%%%%%%%%%%%%%%%%%%%%%%%%%%%%%%%%%%%%%%%%%%%%%%%%%%%%%%%%%%%%%%%%%
%
%
&  \geq  &    \frac{  (2l_1-4k-5)        }{    ( 2 l_1 - m - 2 k  -1 )     }    \times  \frac{  (2 l_1 +m +2 k +3  )   }{   (  2 l_1 -m +2 k +3 ) }
\frac{    (l_1+k+1)   }{  (l_1+k+2)  }   > 1
\end{eqnarray*}

%
%\frac{\left( \begin{array}{ccc} l_1          &   m  &  l_1+2k+1\\  1          &   -1    &    0\end{array}\right)^2       }
%{  \frac{\left( \begin{array}{ccc} l_1          &   m  &  l_1+2k+1\\  1          &   -1    &    0\end{array}\right)^2       }
% }
%\frac{    (l_1-k-1)(l_1+k+1)   }{  (l_1+k+2)(l_1-k)  }
%
%
%
\textbf{Case 2.} If $m$ is odd, the same method is modified in the following way.
Then for $l_1-m+1\leq l_3 \leq l_3+m-1$  in $g^{l_3~-m_1+m}_{l_1m_1 m~-m}$  the summation $L:=l_1+m+l_3$ is $L=l_1+m+l_1\pm j=2l_1+m\pm j$ for some $j\in\mathbb{N}\cup\{0\}$. since $m$ is odd then, the structure constant
$g^{l_3~-m_1+m}_{l_1m_1 m~-m}$ is nonzero if $j=\pm 2k$. Moreover
\begin{eqnarray*}
l_1-m+1 \leq l_1-2k
\end{eqnarray*}
or equivalently $0\leq k\leq\frac{m-1}{2}=[\frac{m-1}{2}]$. Note that for $k=0$ we have
\begin{eqnarray*}
\Big( g^{l_1~-m_1+m}_{l_1m_1m~-m} \Big)^2 \Big( l_1(l_1+1)  -   (l_1)(l_1+1)  \Big)=0
\end{eqnarray*}
which implies that $0< k\leq\frac{m-1}{2}$.
The summation appeared in \eqref{MC with str cts} is symmetric with respect to the new variable $k$. In fact we can write
\begin{eqnarray*}
MC(e_{l_1m_1},e_{m~-m}) &=&   \sum_{k=1}^{\frac{m-1}{2}}  \Big( g^{l_1-2k~-m_1+m}_{l_1m_1m~-m} \Big)^2 \Big( l_1(l_1+1)  -   (l_1-2k)(l_1-2k+1)  \Big) \\
&& + \Big( g^{l_1+2k~-m_1+m}_{l_1m_1m~-m} \Big)^2 \Big( l_1(l_1+1)  -   (l_1+2k)(l_1+2k+1)  \Big).
\end{eqnarray*}
We will prove that each summand is absolutely positive or equivalently for any $1\leq k\leq \frac{m-1}{2}$
\begin{equation*}
f(l_1,m_1,m,k):=\frac{  \Big( g^{l_1-2k~-m_1+m}_{l_1m_1m~-m} \Big)^2 \Big( l_1(l_1+1)  -   (l_1-2k)(l_1-2k+1)  \Big)  }
{\Big( g^{l_1+2k~-m_1+m}_{l_1m_1m~-m} \Big)^2 \Big( -l_1(l_1+1)  +   (l_1+2k)(l_1+2k+1)  \Big)}>1
\end{equation*}
Following the previous part, we prove even more i.e.
\begin{eqnarray*}
1< f(l_1,m_1,m,1) &<&  \dots< f(l_1,m_1,m,k) <  f(l_1,m_1,m,k+1) \\
&<&\dots <  f(l_1,m_1,m,\frac{m-1}{2}).
\end{eqnarray*}
This is equivalent with showing that  $f(l_1,m_1,m,1)>1$ and
\begin{equation*}
1< \frac{  f(l_1,m_1,m,k+1)}{  f(l_1,m_1,m,k)}.
\end{equation*}
for any $1\leq k< k+1\leq \frac{m-1}{2}$.\\
Using  equations \eqref{3j Edmond j1 m1}, \eqref{3j Edmond j1 1 j2 -1}, \eqref{Dowker's formula for str cts} and the fact that $3\leq m\leq m_1$ we observe that 
\begin{eqnarray*}
&& f(l_1,m_1,m,1)\\
&=& \frac{  \Big( g^{l_1-2~-m_1+m}_{l_1m_1m~-m} \Big)^2 \Big( l_1(l_1+1)  -   (l_1-2)(l_1-1)  \Big)  }
{\Big( g^{l_1+2~-m_1+m}_{l_1m_1m~-m} \Big)^2 \Big(  - l_1(l_1+1)  +   (l_1+2)(l_1+3)  \Big)}\\
&=&\frac{  2l_1-3}{2l_1+5} \frac{
\left( \begin{array}{ccc} l_1          &   m  &  l_1-2\\ m_1          &   -m    &    -m_1+m \end{array}   \right)^2
\left(  \begin{array}{ccc}l_1          &   m  &  l_1-2\\1          &   -1    &    0\end{array}\right)^2
}{
\left( \begin{array}{ccc} l_1          &   m  &  l_1+2\\ m_1          &   -m    &    -m_1+m \end{array}\right)^2
\left(\begin{array}{ccc}  l_1          &   m  &  l_1+2\\1             &   -1    &    0      \end{array}\right)^2}
\frac{  2l_1-1}{2l_1+3}   \\
&=&   \frac{ (l_1 + m_1 - m  -1 )  ( l_1 + m_1 -m  )  (l_1 + m_1 - m  +1  )  ( l_1 + m_1 -m +2  )
 }{               (l_1  - m_1 + m -1  )  ( l_1 -m_1 + m    )  ( l_1 -m_1 +  m + 1  )( l_1  -m_1 + m  +2  )   }\\
&&\times \frac{ ( 2 l_1 + m  +2 )^{2}  ( 2 l_1 + m )^{2}  (2 l_1 -3 )  (2 l_1 -1 )   }{   (2 l_1 -m + 2 )^{2}  (2 l_1 +5) (2 l_1 +3) (2 l_1 -m  )^{2}  }\\
&>&  \frac{ ( 2 l_1 + m  +2 )^{2}  ( 2 l_1 + m )   }{   (2 l_1 -m + 2 )^{2}  (2 l_1 +5)   }
\geq  \frac{ ( 2 l_1 + m  +2 )  ( 2 l_1 + m )   }{   (2 l_1 -m + 2 )^{2}     }\\
&>&   \frac{ ( 2 l_1 + m  +2 )     }{   (2 l_1 -m + 2 )     }  >1
\end{eqnarray*}
Again using \eqref{3j Edmond j1 m1} and \eqref{3j Edmond j1 1 j2 -1}  and the fact that $2\leq m\leq m_1$ we get
\begin{eqnarray*}
&&  \frac{  f(l_1,m_1,m,k+1)}{  f(l_1,m_1,m,k)}\\
&=&    \frac{  (l_1 + m_1 -m - 2k - 1    )   ( l_1 + m_1 -m   -2 k   ) (  2 l_1 + m  -2 k )^{2}   ( 2l_1 +m +2 k +2 )^{2}    }{    ( l_1  - m_1  +m - 2 k - 1 )  (  l_1 - m_1 + m - 2 k  ) ( 2 l_1 -m  - 2 k )^{2}( 2 l_1  -m + 2k + 2 )^{2} }\\
&&\times \frac{   ( l_1 + m_1 - m + 2k + 1 )  ( l_1  + m_1   - m + 2k + 2 )      }{        (  l_1 - m_1 + m + 2 k +1  )  ( l_1 - m_1 + m   +2 k + 2 )       } \times \frac{  (2l_1+4k+1)(2l_1-4k-3)  }{  (2l_1 -4k+1)( 2l_1+4k+5) }     \\
&&\times  \frac{  (2l_1-2k-1)(2l_1+2k+1)  }{  (2l_1+2k+3)(2l_1-2k+1)}  \\
%
%======================================================
%
&>&    \frac{  1    }{     ( 2 l_1 -m  - 2 k ) ( 2 l_1  -m + 2k + 2 ) }
\frac{  (2l_1+4k+1)(2l_1-4k-3)  }{  (2l_1 -4k+1) }     \\
&&\times  \frac{  (2l_1-2k-1)(2l_1+2k+1)  }{  (2l_1+2k+3)}  >1
\end{eqnarray*}
which completes the proof or part \textbf{i}.

%
%
%%%%%%%%%%%%%%%%%%%%%%%%%%%%%%%%%          Proof of part 1         %%%%%%%%%%%%%%%%%%%%%%%%%%%%%%%%%%%%%%%%%%%
%
%
Proof of part \textbf{ii.} The proof follows the same method as in part one, utilizing equation \eqref{3j Pres j1 1} instead of \eqref{3j Edmond j1 m1}. However, a quick proof for the case $l_2=1$ is as follows.

For $l_2=2$ we note that  the criterion \eqref{MC with str cts} reduces to 
\begin{eqnarray*}
&&  {MC}(  e_{l_11} , e_{2~1} )  = \sum_{l_3=  l_1- 1  }^{l_1+ 1}       \Big( g^{l_3~-2}_{l_1 1~ 2~1} \Big)^2 \Big( l_1(l_1+1)  -   l_3(l_3+1)  \Big)\\
&=&        \Big( g^{l_1- 1 ~ -2}_{l_1 1~ 2~1} \Big)^2 \Big( l_1(l_1+1)  -   l_1(l_1-1)  \Big)\\
&& +         \Big( g^{l_1+ 1 ~ -2}_{l_1 1~ 2~1} \Big)^2 \Big( l_1(l_1+1)  -   (l_1+ 1)(l_1+ 2)  \Big).
\end{eqnarray*}
We will show that
\begin{eqnarray*}
&&  \Big( g^{l_1- 1 ~ -2}_{l_1 1~ 2~1}  \Big)^2 \Big( l_1(l_1+1)  -   l_1(l_1-1)  \Big)    >\\
&& |       \Big( g^{l_1+ 1 ~ -2}_{l_1 1~ 2~1} \Big)^2 \Big( l_1(l_1+1)  -   (l_1+ 1)(l_1+ 2)  \Big)   |
\end{eqnarray*}
or equivalently
\begin{eqnarray*}
I&:=&   \frac{   \Big( g^{l_1- 1 ~ -2}_{l_1 1~ 2~1}  \Big)^2    }{
\Big( g^{l_1+ 1 ~ -2}_{l_1 1~ 2~1}  \Big)^2 }  \times  \frac{\Big( l_1(l_1+1)  -   l_1(l_1-1)  \Big)}{\Big(- l_1(l_1+1)  +   (l_1+ 1)(l_1+ 2)  \Big)}\\
&=&  \frac{   \Big( g^{l_1- 1 ~ -m_1+2}_{l_1m_1 2~-2} \Big)^2    }{
\Big( g^{l_1+ 1 ~ -m_1+2}_{l_1m_1 2~-2} \Big)^2 } \times  \frac{l_1}{l_1+1}>1
\end{eqnarray*}
Using equation \eqref{Dowker's formula for str cts}, \eqref{3j Pres j1 1} and \eqref{3j Edmond j1 1 j2 -1}   after some simplifications  we observe that
\begin{eqnarray*}
&&    \frac{   \Big( g^{l_1- 1 ~ -2}_{l_1 1~ 2~1}  \Big)^2    }{
\Big( g^{l_1+ 1 ~ -2}_{l_1 1~ 2~1}  \Big)^2 }
=\\
&&  =  \frac{             (  2l_1-1  ) 
\left(
\begin{array}{ccc}
l_1    &   2  &  l_1-1\\
1          &   1  &  -2
\end{array}\right)
\left(
\begin{array}{ccc}
l_1    &   2  &  l_1-1\\
1      &   -1   &   0
\end{array}\right)                        }{ 
(  2l_1+3  )
\left(
\begin{array}{ccc}
l_1    &   2  &  l_1+1\\
 1          &  1  &   -2
\end{array}\right)
\left(
\begin{array}{ccc}
l_1    &   2  &  l_1+1\\
1      &   -1   &   0
\end{array}\right)                      } \\
&&=  \frac{(  2l_1-1  )}{(  2l_1+3  )}    \times  \frac{  (l_1+3)(2 l_1 +3 )^2  (l_1 -1  )   }{  (l_1-2 )  (l_1 +2 )  (2  l_1 -1 )^2}\\
&&=   \frac{  (l_1+3)(2 l_1 +3 )  (l_1 -1  )   }{  (l_1-2 )  (l_1 +2 )  (2  l_1 -1 )}
\end{eqnarray*}
As a result we have

\begin{eqnarray*}
I=   \frac{  (l_1+3)(2 l_1 +3 )  (l_1 -1  )   }{  (l_1-2 )  (l_1 +2 )  (2  l_1 -1 )}
 \times   \frac{l_1}{l_1+1}
>  \frac{  (2 l_1 +3 )     }{    (2  l_1 -1 )}
 \times   \frac{l_1}{l_1+1}  >1.
\end{eqnarray*}
which completes the proof.
\end{proof}
%
%
%%%%%%%%%%%%%%%%%%%%%%%%%%%%%%%%%%%%%%        end proof of the theorem       %%%%%%%%%%%%%%%%%%%%%%%%%%%%%%%%
%
\begin{cor}
Let $x=(x_1,\dots,x_n)\in\mathbb{R}^n$ be an arbitrary vector and $m_1^\prime<m_2^\prime\dots<m_n^\prime$ be integers. For $2\leq m_1\leq l_1$ and $2\leq m\leq m_1$ using proposition \ref{Prop Mis Curv linearity} for  the Misiolek criterion we get
\begin{eqnarray*}
&&  MC(e_{l_1 m_1} ,e_{m-m} +\sum_{j=1}^n x_je_{l_j^\prime m_j^\prime})  \\
&=&  MC(e_{l_1 m_1} ,e_{m-m} ) +\sum_{j=1}^n |x_j|^2MC(e_{l_1 m_1} , e_{l_j^\prime m_j^\prime})\\
&\geq&  MC(e_{l_1 m_1} ,e_{m-m} ) + ||x||^2\Lambda
\end{eqnarray*}
where
\[
\Lambda=\min \{MC(e_{l_1 m_1} , e_{l_j^\prime m_j^\prime})\}_{j=1,\dots n}.  
\]
As a result, we observe  that by replacing $e_{m-m}$ with
$e_{m-m}+\sum_{j=1}^n x_j e_{l_j^\prime m_j^\prime}$, where $||(x_1,\dots,x_n)||$ is sufficiently small, conjugate points still exist.
\end{cor}
\begin{Rem}
It is expected that the previous theorem could be generalized for the case where $2 \leq m \leq 2m_1 - 2$. However, it seems that one needs to consider even more cases, such as $m = m_1 + 1$, $m_1 + 2 \leq m \leq 2m_1 - 3$, and $m = 2m_1 - 2$, and write the computations for both odd and even cases. Moreover, in the context of the previous theorem, the summation that appears in the Misiolek criterion is not necessarily  symmetric. 
A similar discussion might be stated for $MC(e_{l_1~ 1} ,e_{l_2~1}).$

Keeping the statement of the second part of the corollary \ref{Cor-MC-l1=1} about $e_{1~1}$ in mind, the sole exception that we can not guarantee the existence  of conjugate points is $e_{2~\pm1}.$

Of course, we know that the Misiolek criterion for zonal spherical harmonics on $\cD^s_{vol}(\Sph^2)$ is always non-positive \cite{Tauchi-Yoneda}.
\end{Rem}
%
%
%
%
%
%%%%%%%%%%%%%%%%%%%%%%%%%%%%%%%%%%%%      Section  Conjugate points at the presence of the Coriolis force      %%%%%%%%%%%%%%%%%%%%%%%%%%%%%%%%%%%%%%%
%%%%%%%%%%%%%%%%%%%%%%%%%%%%%%%%%%%%      Section  Conjugate points at the presence of the Coriolis force      %%%%%%%%%%%%%%%%%%%%%%%%%%%%%%%%%%%%%%%
%%%%%%%%%%%%%%%%%%%%%%%%%%%%%%%%%%%%      Section  Conjugate points at the presence of the Coriolis force      %%%%%%%%%%%%%%%%%%%%%%%%%%%%%%%%%%%%%%%
%
%
%

\section{ Conjugate points and  the impact of the  Coriolis force}

In this section, we investigate the impact of the Coriolis force on conjugate points along two key solutions of the QGS \eqref{eq geodesic-equation of central extension reduced} which are generated by spherical harmonics. Initially, we present the Misiolek criterion, defined in (\ref{eqn Misiolek Curvature on S2}), utilizing real structure constants.

In the absence of the Coriolis effect, the Misiolek Criterion cannot ensure the existence of conjugate points along vector fields generated by zonal spherical harmonics \cite{Tauchi-Yoneda}. In this section, we will first try to prove the existence of conjugate points  along zonal spherical harmonics in the presence of the Coriolis effect.

% When dealing with zonal spherical harmonics, we begin by calculating Jacobi fields in the direction of individual (zonal) spherical harmonics. %This analysis ensures the existence of conjugate points only for the wave number $l_1=1$.

 We will prove that by varying rotation speeds and direction of rotation (i.e. suitable values for $a\in\mathbb{R}$), conjugate points can occur along any $e_{l_1~0}$ with $l_1=2k+1\in\N$.

On the other hand, Rossby Haurwitz waves (RHW for short), widely used in meteorology, offer a time-dependent class of solutions for the QGS \eqref{eq geodesic-equation of central extension reduced}. Benn \cite{Benn2021} adopted  the Misiolek criterion  for these time-dependent solutions. We observe that the Coriolis effect stabilizes the system, generating conjugate points that wouldn't appear without it.
\subsection{Misiolek criterion at the presence of the Coriolis force}\label{sub sect. CP non-zonal}
Following section \ref{Sect. CP S.H}, first we compute the Misiolek criterion (\ref{eqn Misiolek Curvature on S2})  according to the structure constants.
%
%
%%%%%%%%%%%%%%%%%%%%%%%%%%%%%%%%%%%%%%%%%%%%%%%%%%%%          Prop Mis Curv e rot         %%%%%%%%%%%%%%%%%%%%%%%%%%%%%%%%%%%%%%%%%%%%%%%%%%%%%%%%%%555
%
%
\begin{prop}\label{Prop Mis Curv e rot}
For  $e_{l_1m_1}$ and $e_{l_2m_2}$  %and $\phi=\mu=\sqrt{\frac{4\pi}{3}}Y_{10}$
we have
\begin{eqnarray}\label{MC for rot s.h.}
\widehat{MC}\Big( {(e_{l_1m_1},a),(e_{l_2m_2},b)}  \Big)  &=& \sum_{l_3,m_3}( g^{l_3-m_3}_{l_1m_1l_2m_2})^2 \Big(   l_1(l_1+1)  -l_3(l_3+1)   \Big) \nonumber\\
&&  -  m_1^2\delta^{m_1}_{m_2}  \delta^{l_1}_{l_2}   + (-1)^{m_2}a m_2 g^{l_2-m_2}_{l_1m_1l_2m_2}
\end{eqnarray}
\end{prop}

\begin{proof}
According to proposition \ref{Prop Mis Curv e}, for  $f=Y_{l_1m_1}$, $g=Y_{l_2m_2}$  we have
\begin{eqnarray*}
&& \langle \Delta \{f,g\}  ,  \{f,g\}  \rangle  - \langle  \{\Delta f,g\}  ,  \{f,g\}  \rangle \\
&=& \sum_{l_3,m_3}( g^{l_3-m_3}_{l_1m_1l_2m_2})^2 \Big(   l_1(l_1+1)  -l_3(l_3+1) \Big)
\end{eqnarray*}
Moreover
\begin{eqnarray*}
\langle  \{\mu,f\}  ,  g  \rangle^2  &=&   \Big|  -im_1 \langle Y_{l_1m_1}  ,  Y_{l_2m_2} \rangle  \Big|^2  =   m_1^2 \delta^{m_1}_{m_2}  \delta^{l_1}_{l_2}
\end{eqnarray*}
and
\begin{eqnarray*}
\langle \{\mu, \{f,g\}\}  ,  g  \rangle  &=& - im_3 G^{l_3m_3}_{l_1m_1l_2m_2}   \langle Y_{l_3m_3}  ,  Y_{l_2m_2}  \rangle  \\
&=&
- i m_3 G^{l_3m_3}_{l_1m_1l_2m_2}  \delta^{m_3}_{m_2}  \delta^{l_3}_{l_2}\\
&=&     -i m_2 G^{l_2m_2}_{l_1m_1l_2m_2}\\
&=&     i^2 (-1)^{m_2}m_2 g^{l_2  -m_2}_{l_1m_1l_2m_2}\\
&=&     -(-1)^{m_2} m_2 g^{l_2-m_2}_{l_1 m_l l_2m_2}
\end{eqnarray*}
The result is followed by considering equation (\ref{eqn Misiolek Curvature on S2}) and substituting the above results.
\end{proof}
%
%
%%%%%%%%%%%%%%%%%%%%%%%%%%%%%%%%%%%%%%%%%%%%    end proof of prop          %%%%%%%%%%%%%%%%%%%%%%%%%%%%%%%%%%%%%%%%%%%%%%
%
%
As a result of the theorem \ref{Theorem MC S.H.} and \eqref{MC for rot s.h.} we see that for $2\leq m\leq m_1$ 
\begin{eqnarray*}\label{MC for rot s.h.}
{MC}\Big( {(e_{l_1m_1},a),(e_{m~-m},b)}  \Big)  ={MC}\Big( e_{l_1m_1},e_{m~-m}  \Big)>0   
\end{eqnarray*}
and for any $2\leq l_2<l_1$
\begin{eqnarray*}\label{MC for rot s.h.}
{MC}\Big( {(e_{l_1~1},a),(e_{l_2~1},b)}  \Big)  ={MC}\Big( e_{l_1~1},e_{l_2~1}  \Big)>0   
\end{eqnarray*}
which means that after considering the Coriolis force, the conjugate points suggested by theorem \ref{Theorem MC S.H.} still appear. However, the term 
\[
(-1)^{m_2}a m_2 g^{l_2-m_2}_{l_1m_1l_2m_2}
\]
is nonzero when $m_1 + m_2 - m_2 = m_1 = 0.$ Consequently, this term could assist us only in identifying conjugate points along zonal spherical harmonics.

%
%

%
%
%
%%%%%%%%%%%%%%%          Conjugate points along zonal spherical harmonics at the presence of the Coriolis force            %%%%%%%%%%%%%%%%%%
%%%%%%%%%%%%%%%          Conjugate points along zonal spherical harmonics at the presence of the Coriolis force            %%%%%%%%%%%%%%%%%%
%%%%%%%%%%%%%%%          Conjugate points along zonal spherical harmonics at the presence of the Coriolis force            %%%%%%%%%%%%%%%%%%
%
%
%
\subsection{Conjugate points along zonal spherical harmonics}
In this section will prove that by varying rotation speeds  and direction of rotation (rotation rate)  which is governed by the parameter $a$, conjugate points can occur along any $e_{l_1~0}$ with $l_1=2k+1\in\N$.
%
%
%
%\begin{Rem}
The last term appeared in \eqref{MC for rot s.h.} is nonzero if the first spherical harmonics $e_{l_1m_1}$ is zonal and it is perturbed by a non-zonal wave. More precisely $g^{l_2-m_2}_{l_1 m_l l_2m_2} \neq 0$ if  $m_1+m_2-m_2=m_1=0$. Moreover $m_2g^{l_2-m_2}_{l_1 m_l l_2m_2} \neq 0$ if
$m_1=0$ and $m_2\neq 0$ and in this case
\begin{eqnarray*}
\widehat{MC}\Big( {(e_{l_1~0},a),(e_{l_2m_2},b)}  \Big)  &=& {MC}\Big( e_{l_1~0}, e_{l_2m_2}  \Big)
+ (-1)^{m_2}a m_2 g^{l_2-m_2}_{l_1~0l_2m_2}.
\end{eqnarray*}
Since $m_2 g^{l_2-m_2}_{l_1~0l_2m_2}=-m_2 g^{l_2m_2}_{l_1~0l_2-m_2}$ then we have
\begin{eqnarray*}
\widehat{MC}\Big( {(e_{l_1~0},a),(e_{l_2m_2},b)}  \Big)  &=& {MC}\Big( e_{l_1~0}, e_{l_2m_2}  \Big)
+ (-1)^{m_2}a m_2 g^{l_2-m_2}_{l_1~0l_2m_2}\\
 &=& {MC}\Big( e_{l_1~0}, e_{l_2-m_2}  \Big) + (-1)^{m_2}a (-m_2) g^{l_2m_2}_{l_1~0l_2-m_2}\\
 &=& \widehat{MC}\Big( {(e_{l_1~0},a),(e_{l_2-m_2},b)}  \Big)
\end{eqnarray*}
This last means that it suffices to consider just the case that  $m_2$ is a natural number.
Moreover, the term $m_2 g^{l_2-m_2}_{l_1~0l_2m_2}\neq 0$ if $l_1+2l_2$ is not even. In the other words, the term $m_2 g^{l_2-m_2}_{l_1~0l_2m_2}$ vanishes when $l_1$ is even. In fact if $ l_1=2k + 1$ for some $k\in\mathbb{N}$ then for any $1<l_2$ and $0<|m_2|\leq l_2$ for a suitable choice of $a$ the conjugate points exists. The critical ratios for the parameter $a$ that ensure the positivity of the Misiolek criterion are presented in Table (\ref{my-label}) for various wave numbers with $1 \leq l_1 \leq 5$.

%
%%%%%%%%%%%%%%%%%%%%%%%%%%%%%%%%%%%%%%%%%%%%%%%%%%%%%%%%%%%%%%%%%   Remark   %%%%%%%%%%%%%%%%%%%%%%%%%%%%%%%%%%%%%%%%%
%
\begin{Rem}
Note that $MC(e_{l_1~0},e_{l_2m_2})$ is always non-positive. Consequently, the change in sign in the term $(-1)^{m_2}m_2g^{l_2-m_2}_{l_1m_1l_2m_2}$ is responsible for the suggested form of the inequality suggested by the critical ratio of $a$. More precisely if  $(-1)^{m_2}m_2g^{l_2-m_2}_{l_1m_1l_2m_2}>0$ then, 
\[
a> \frac{-MC(e_{l_10},e_{l_2m_2})}{(-1)^{m_2}m_2 g_{l_10~l_2m_2}^{l_2-m_2}}.
\]
In the case that  $(-1)^{m_2}m_2g^{l_2-m_2}_{l_1m_1l_2m_2}<0$ we have the condition 
\[
a< \frac{-MC(e_{l_10},e_{l_2m_2})}{(-1)^{m_2}m_2 g_{l_10~l_2m_2}^{l_2-m_2}}.
\]
In table (\ref{my-label}), the red numbers are associated with $(-1)^{m_2}m_2g^{l_2-m_2}_{l_1m_1l_2m_2}<0$, and for the other values, we have $(-1)^{m_2}m_2g^{l_2-m_2}_{l_1m_1l_2m_2}>0$.

Intuitively to catch conjugate points, we have to change the direction of rotation and modify the speed of this rotation according to the critical ratio.

In meteorology, the stability of zonal flows on a rotating sphere according to the critical ratios for the rotation rate, has been studied by several authors, e.g. \cite{Bains} and \cite{Sasaki}. The previous results could be considered as geometric counterparts from the perspective of nonlinear stability.
\end{Rem}

As we can see from Table \ref{my-label}, generally, proposing a simple closed form for $g^{l_2-m_2}_{l_1m_1l_2m_2}$ and consequently determining the sign of this term is not easy. However, in the special case that $l_2=m_2$, formula \eqref{3j Edmond j1 m1} implies that
\begin{eqnarray*}
&&  \left(
\begin{array}{ccc}
l_1    &   l_2  &  l_2\\
0          &  l_2  &   -l_2
\end{array}\right)  =   \Big(   \frac{\left(2 {l_1} \right)!^{2}}{\left(2 {l_1} +{l_3} +1\right)! \left(2 {l_1} -{l_3} \right)!}   \Big)^{\frac{1}{2}}.
\end{eqnarray*}
Now the structure constant
\begin{eqnarray*}
&& g^{l_2-l_2}_{l_1 ~0~ l_2l_2} =   \frac{-1}{  \sqrt{4\pi}  } \Big( (2l_1 +1) (2l_2+1)^2  l_1(l_1+1) l_2(l_2+1)  \\
&&  \times   \frac{\left(2 {l_1} \right)!^{2}}{\left(2 {l_1} +{l_3} +1\right)! \left(2 {l_1} -{l_3} \right)!}   \Big)^{\frac{1}{2}}
\left(
\begin{array}{ccc}
l_1    &   l_2  &  l_3\\
1      &  - 1  &   0
\end{array}\right)
\end{eqnarray*}
 Moreover,
\begin{eqnarray*}
&&  \left(
\begin{array}{ccc}
l_1    &   l_2  &  l_3\\
1      &  - 1  &   0
\end{array}\right)  = (-1)^{\frac{J}{2}}  \, (\frac{J}{2}  )!\\
&&\frac{   \Big(   \frac{ (J +1 ) (J -2  l_3 ) (J -2 l_1 ) (J -2 l_2 -1 )}{l_1 (l_1 + 1 ) {l_2} \left(l_2+ 1 \right)}            \times             \frac{\left(J -2 {l_3} \right)! \left(J -2 {l_1} \right)! \left(J -2 {l_2} -2\right)!}{\left(J +1\right)!}  \Big)^{\frac{1}{2}}    }{    2 \left(\frac{J}{2}-{l_3} \right)! \left(\frac{J}{2}-{l_1} \right)! \left(\frac{J}{2}-{l_2} -1\right)!} \nonumber
\end{eqnarray*}
and in the case that $J=l_1+l_2+l_3+1$ is odd, the $3j$ symbols vanishes.
As a result the sign of $ g^{l_2-l_2}_{l_1 ~0~ l_2l_2} $ is $-(-1)^{\frac{l_1+2l_2+1}{2}}$.

%
%
%
%
%
%
%
%
%
%
%%%%%%%%%%%%%%%%%%%%%%%%%%%%%%%%%%%%%%%%%%%%%%%%           Table of critical ratio        %%%%%%%%%%%%%%%%%%%%%%%%%%%%%%%%%%%%
%%%%%%%%%%%%%%%%%%%%%%%%%%%%%%%%%%%%%%%%%%%%%%%%           Table of critical ratio        %%%%%%%%%%%%%%%%%%%%%%%%%%%%%%%%%%%%
%
%
\begin{table}
\caption{The critical ratio $\frac{-MC(e_{l_10},e_{l_2m_2})}{(-1)^{m_2}m_2 g_{l_10~l_2m_2}^{l_2-m_2}}$.}
\label{my-label}
\begin{tabular}{|*{8}{c|}}
\hline
\multicolumn{2}{|c|}{\multirow{3}{*}{Ratio}}
& \multicolumn{6}{c|}{$m_2$} \\
\cline{3-8}
\multicolumn{2}{|c|}{}
& \multicolumn{1}{c|}{1}
& \multicolumn{1}{c|}{2}
& \multicolumn{1}{c|}{3}
& \multicolumn{1}{c|}{4}
& \multicolumn{1}{c|}{5}
& \multicolumn{1}{c|}{6} \\
\multicolumn{2}{|c|}{} &  &  &  &  &  &  \\
\hline
\multirow{4}{*}{$l_1=3$}
& $l_2=2$ &  2.983  & \textcolor{red}{-19.39}  & -- & --  & -- & -- \\ \cline{2-8}
& $l_2=3$ & 12.20   & 30.53 &  \textcolor{red}{-20.35} & --  & -- & -- \\ \cline{2-8}
& $l_2=4$ & 41.51  & 43.87 & 73.68 & \textcolor{red}{-24.43} & -- & -- \\ \cline{2-8}
& $l_2=5$ & 80.5 & 77.62 & 78.54  & 192.4 & \textcolor{red}{-31.31} & -- \\ \hline
\multirow{4}{*}{$l_1=5$}
& $l_2=2$ & 0  & 0  & -- & --  & -- &--\\ \cline{2-8}
& $l_2=3$ & 19.41   & \textcolor{red}{-71.19} & 170.4 & 0  & 0 & 0 \\ \cline{2-8}
& $l_2=4$ & 45.64 & 269.0 & \textcolor{red}{-60.40} & 125.4 & 0 & 0 \\ \cline{2-8}
& $l_2=5$ & 101.6 & 226.5 & \textcolor{red}{-616.9}  & \textcolor{red}{-72.31} & 123.5 & 0 \\ \hline
\multirow{5}{*}{$l_1=7$}
& $l_2=2$ & 0 & 0  & -- & --  & -- & -- \\ \cline{2-8}
& $l_2=3$ & 0   & 0 & 0 & --  & -- & -- \\ \cline{2-8}
& $l_2=4$ & 71.66 & \textcolor{red}{-205.5} & 276.7 & \textcolor{red}{-1279} & -- & -- \\ \cline{2-8}
& $l_2=5$ & 127.8 & 4792 &  \textcolor{red}{-171.4}  & 182.1 & \textcolor{red}{-713.5} & -- \\ \hline\cline{2-8}
& $l_2=6$ & 234.1 & 881.2 & \textcolor{red}{-475.7}  & \textcolor{red}{-245.1} & 175.2& \textcolor{red}{-569.9}  \\ \hline\cline{2-8}
\end{tabular}
\end{table}

%
%
%
%%%%%%%%%%%%%%%%%%%%%%%%%%%%%%%%%%%%%%         Rossby-haurwitz waves       %%%%%%%%%%%%%%%%%%%%%%%%%%%%%%%%
%%%%%%%%%%%%%%%%%%%%%%%%%%%%%%%%%%%%%%         Rossby-haurwitz waves       %%%%%%%%%%%%%%%%%%%%%%%%%%%%%%%%
%%%%%%%%%%%%%%%%%%%%%%%%%%%%%%%%%%%%%%         Rossby-haurwitz waves       %%%%%%%%%%%%%%%%%%%%%%%%%%%%%%%%
%
%
%
\subsection{Conjugate points along complex Rossby-Haurwitz waves}
Following the formalism of \cite{Skiba}, Rossby-Haurwitz waves are defined as follows
\begin{equation}\label{RH waves}
\Psi(\lambda,\mu,t)= \sum_{m=-l}^l \psi_{lm} Y_{lm}(\lambda-\omega t,\mu) -C\mu 
\end{equation}
where $\omega,C\in\mathbb{R}$ and $\psi_{lm}\in\mathbb{C}$ are constants.
Note that
\begin{eqnarray*}
\Delta\Psi  &=&  \Delta\sum_{m=-l}^l \psi_{lm} Y_{lm}(\lambda-\omega t,\mu) -C\Delta\mu\\
&=&-l(l+1)\sum_{m=-l}^l \psi_{lm} Y_{lm}(\lambda-\omega t,\mu)+ 2C\mu
\end{eqnarray*}
and $\partial_t(\Delta\Psi -\alpha^2\Psi)=(\omega l(l+1) + \omega\alpha^2)\frac{\partial \Psi}{\partial \lambda}$. Moreover we have
\begin{eqnarray*}
&&   \{ \Delta\Psi -a\mu,\Psi \}  =  \{  -l(l+1)\sum_{m=-l}^l \psi_{lm} Y_{lm}(\lambda-\omega t,\mu)+ 2C\mu -a\mu,\Psi  \}\\
&&   =\{  -l(l+1)\Big[\sum_{m=-l}^l \psi_{lm} Y_{lm}(\lambda-\omega t,\mu) - C\mu + C\mu\Big]+ 2C \mu -a\mu,\Psi  \}\\
&&   =\{  -l(l+1)\Psi -l(l+1)C \mu+ 2C\mu -a\mu,\Psi  \}\\
&&   = -l(l+1) \{ \Psi,\Psi \} + (l(l+1)C-2C + a) \frac{\partial \Psi}{\partial \lambda}\\
&&   = (l(l+1)C-2C + a) \frac{\partial \Psi}{\partial \lambda}.
\end{eqnarray*}
As a result $\Psi(\lambda,\mu,t)$ is a solution of the quasi-geostrophic equation \eqref{eq geodesic-equation of central extension reduced} if $\partial_t(\Delta\Psi -\alpha^2\Psi) = \{\Delta\Psi -a\mu,\Psi\}$ or equivalently 
\begin{equation}\label{eqn RHW sol}
\omega l(l+1) + \omega\alpha^2=l(l+1)C-2C + a. 
\end{equation}
%or
%
%$\omega  = \frac{(l(l+1)C-2C + a)}{l(l+1) +\alpha^2}.$ 
%
For simplicity suppose that $\Psi$ has the form
\begin{equation*}
\Psi(\lambda,\mu,t) = A Y_{l_1m_1}(\lambda-\omega t,\mu) -C \mu     
\end{equation*}
with $m_1\neq 0$.
%
%
%%%%%%%%%%%%%%%%%%%%%%%%%%%%%%%%%%%%%%%%%%%%%   Prop Mis Curv RHW    %%%%%%%%%%%%%%%%%%%%%%%%%%%%%%%%%%%
%
%
\begin{prop}\label{Prop MC RHW}
With the above assumptions the followings hold true.\\
\textbf{1.} For $0\leq |m_2|\leq l_2$ we have
\begin{eqnarray*}
MC( \nabla^\perp \Psi , e_{m~-m} ) =   |A|^2{MC}(  e_{l_1m_1} , e_{l_2m_2} ) + C^2m_2^2 (2-l_2(l_2+1) ).
\end{eqnarray*}
\textbf{2.} For $(\nabla^\perp \Psi ,a) ,   ( e_{l_2m_2},b) \in\hat\cg$ the Misiolek criterion is given by
\begin{eqnarray*}
\widehat{MC}\Big( (\nabla^\perp \Psi ,a) ,   (e_{l_2m_2},b) \Big) &=&   |A|^2{MC}(  e_{l_1m_1} , e_{l_2m_2} ) + C^2m_2^2 (2-l_2(l_2+1) )\\
&&  -|A|^2m_1^2  \delta^{l_1}_{l_2}\delta^{m_1}_{m_2} -a m_2^2C.
\end{eqnarray*}
\end{prop}
%
%
%%%%%%%%%%%%%%%%%%%%%%%%%%%%%%%%%%%%%%%%                  proof RHW             %%%%%%%%%%%%%%%%%%%%%%%%%%%%%%%%%%%%
%%%%%%%%%%%%%%%%%%%%%%%%%%%%%%%%%%%%%%%%                  proof RHW             %%%%%%%%%%%%%%%%%%%%%%%%%%%%%%%%%%%%
%
%
\begin{proof}
\textbf{1.} Since $\mu =\sqrt{\frac{4\pi}{3}} Y_{1~0}$, and $ Y_{l_1m_1}(\lambda-\omega t,\mu) =e^{-im_1\omega t}  Y_{l_1m_1}(\lambda,\mu) $, using proposition \ref{Prop Mis Curv linearity} part \textbf{ii} we get
\begin{eqnarray*}
&&  {MC}( \nabla^\perp \Psi , e_{l_2m_2} ) =  |Ae^{-im_1\omega t}|^2{MC}(  e_{l_1m_1} , e_{l_2m_2} ) + \frac{4\pi}{3}C^2 {MC}(  e_{1~0} , e_{l_2m_2} ).
\end{eqnarray*}
On the other hand,  using corollary \ref{Cor-MC-l1=1} and the fact that $g^{1~0}_{l_2m_2l_2-m_2}=(-1)^{m_2}m_2\sqrt{\frac{3}{4\pi}}$ (see also \cite{Arak-Sav} equation A14) we have
\begin{eqnarray*}
{MC}(  e_{1~0} , e_{l_2m_2} )   & = & \sum_{l_3=  l_2-1 +1  }^{l_2+1-1}       \Big( g^{l_3-m_2}_{1~0~l_2m_2} \Big)^2 \Big( 2  -   l_3(l_3+1)  \Big)\\
&=& \Big( g^{l_2-m_2}_{1~0~l_2m_2} \Big)^2 \Big( 2  -   l_2(l_2+1)  \Big)\\
&=&  \Big( g^{1~0}_{l_2m_2l_2-m_2}   \Big)^2 \Big( 2  -   l_2(l_2+1)  \Big)\\
&=&  m_2^2\frac{3}{4\pi} \Big( 2  -   l_2(l_2+1)  \Big).
\end{eqnarray*}
As a result we get
\begin{eqnarray*}
&&  {MC}( \nabla^\perp \Psi , e_{l_2m_2} ) =  |A|^2{MC}(  e_{l_1m_1} , e_{l_2m_2} ) + C^2m_2^2( 2-l_2(l_2+1) ).
\end{eqnarray*}
\textbf{2.} 
First note that
\begin{eqnarray*}
\{\mu,\Psi\}  &=&   \{\mu,A Y_{l_1m_1}(\lambda-\omega t,\mu) -C \mu\}\\
&=&   A e^{-im_1\omega t} \{\mu, Y_{l_1m_1}(\lambda,\mu) \}\\
&=&   A e^{-im_1\omega t} (-\frac{\partial \mu}{\partial \mu}  ) \frac{\partial}{\partial\lambda} Y_{l_1m_1} \\
&=&   - im_1A e^{-im_1\omega t} Y_{l_1m_1} 
\end{eqnarray*}
Moreover the term $\langle\{\mu,\Psi\},Y_{l_2m_2}\rangle^2$ in \eqref{eqn Misiolek Curvature on S2} can be calculated as follows
\begin{eqnarray*}
\langle\{\mu,\Psi\},Y_{l_2m_2}\rangle^2  &=&  |- im_1A e^{-im_1\omega t} |^2 \langle  Y_{l_1m_1}   ,  Y_{l_2m_2} \rangle^2\\
&=&   m_1^2|A|^2\delta^{l_1}_{l_2}\delta^{m_1}_{m_2}.
\end{eqnarray*}
On the other hand we have
\begin{eqnarray*}
\{\Psi,Y_{l_2m_2}\}  &=&   \{A  Y_{l_1m_1}(\lambda-\omega t,\mu) -C \mu   ,  Y_{l_2m_2}\}\\
&=&   A e^{-im_1\omega t} \{ Y_{l_1m_1} , Y_{l_2m_2} \}  -  C \sqrt{\frac{4\pi}{3}}  \{ Y_{1~0} , Y_{l_2m_2} \} \\
&=&   \Big( A e^{-im_1\omega t} G^{l_3m_3}_{l_1m_1l_2m_2}  -  C \sqrt{\frac{4\pi}{3}}  G^{l_3m_3}_{1~0~l_2m_2}\Big)  Y_{l_3m_3}  
\end{eqnarray*}
and
\begin{eqnarray*}
\{\mu,    \{\Psi,Y_{l_2m_2}\}  \}  &=& -im_3\Big( A e^{-im_1\omega t} G^{l_3m_3}_{l_1m_1l_2m_2}  -  C \sqrt{\frac{4\pi}{3}}  G^{l_3m_3}_{1~0~l_2m_2}\Big) Y_{l_3m_3} 
\end{eqnarray*}
which implies that
\begin{eqnarray*}
\langle \{\mu,    \{\Psi,Y_{l_2m_2}\} , Y_{l_2m_2} \rangle  &=& -im_3\Big( A e^{-im_1\omega t} G^{l_3m_3}_{l_1m_1l_2m_2}  -  C \sqrt{\frac{4\pi}{3}}  G^{l_3m_3}_{1~0~l_2m_2}\Big) \delta^{l_3}_{l_2}\delta^{m_3}_{m_2}\\
&=&   -im_2\Big( A e^{-im_1\omega t} G^{l_2m_2}_{l_1m_1l_2m_2}  -  C \sqrt{\frac{4\pi}{3}}  G^{l_2m_2}_{1~0~l_2m_2}\Big) \\
&=&   -im_2\Big(   -  C(-i)(-1)^{m_2} \sqrt{\frac{4\pi}{3}}  g^{l_2 ~ -m_2}_{1~0~l_2m_2}\Big) \\
%
%&=&   -  C(-i)^2m_2(-1)^{m_2} \sqrt{\frac{4\pi}{3}}     g^{l_2 ~ -m_2}_{1~0~l_2m_2} \\
%
&=&     C m_2(-1)^{m_2} \sqrt{\frac{4\pi}{3}}     g^{1~0}_{l_2m_2 l_2~-m_2}    =  C m_2^2.
%
%
%&=&     C m_2^2(-1)^{2m_2} \sqrt{\frac{4\pi}{3}}   \sqrt{\frac{3}{4\pi}}  =  C m_2^2.
\end{eqnarray*}
Now using formula \eqref{eqn Misiolek Curvature on S2} and the above facts we see that
\begin{eqnarray*}
\widehat{MC}\Big( (\nabla^\perp \Psi ,a) ,   (e_{l_2m_2},b) \Big) &=&   |A|^2{MC}(  e_{l_1m_1} , e_{l_2m_2} ) + C^2m_2^2 (2-l_2(l_2+1) )\\
&&  -|A|^2m_1^2 \delta^{l_1}_{l_2}\delta^{m_1}_{m_2} -a m_2^2C
\end{eqnarray*}
which completes the proof.
\end{proof}
%
%
%%%%%%%%%%%%%%%%%%%%%%%%%%%%%%%%%%%%%%%%%%%%%%%%%%  Corollary RHW CP    %%%%%%%%%%%%%%%%%%%%%%%%%%%%%%%%%%%%%%%%%%%%%%%%%%
%%%%%%%%%%%%%%%%%%%%%%%%%%%%%%%%%%%%%%%%%%%%%%%%%%  Corollary RHW CP    %%%%%%%%%%%%%%%%%%%%%%%%%%%%%%%%%%%%%%%%%%%%%%%%%%
%
%
\begin{cor}
Suppose that $l_1\neq 0$,  $l_2=m$, $m_2=-m$ with $2\leq m\leq m_1$. As a result of theorem \ref{Theorem MC S.H.} and proposition \ref{Prop MC RHW} part 1 we have 
\begin{eqnarray}\label{eq A/C quotient}
&&  {MC}( \nabla^\perp \Psi , e_{m~-m} )>0  \iff   \frac{|A|^2}{C^2}         >        \frac{      m^2( m(m+1)-2)        }{      {MC}(  e_{l_1m_1} , e_{m~-m} )}.
\end{eqnarray}
At the presence of the Coriolis force, suppose that $a=-KC$ where $K$ is a positive real number. Then, proposition \ref{Prop MC RHW} part 2 implies that 
\begin{eqnarray*}
\widehat{MC}\Big( (\nabla^\perp \Psi ,a) ,   (e_{m~-m},b) \Big) &=&   |A|^2{MC}(  e_{l_1m_1} , e_{m~-m} ) + C^2m^2 (2-m(m+1) )\\
&&   +K m^2C^2\\
&>&    {MC}( \nabla^\perp \Psi , e_{m~-m} )
\end{eqnarray*}
and $\widehat{MC}\Big( (\nabla^\perp \Psi ,a) ,   (e_{m~-m},b) \Big)>0$  if and only if 
\begin{eqnarray}\label{eq A/C quotient Coriolis}
\frac{|A|^2}{C^2}         >        \frac{      m^2( m(m+1)-2-K)        }{      {MC}(  e_{l_1m_1} , e_{m~-m} )}.
\end{eqnarray}
As another special case,  using  proposition \ref{Cor-MC-l1=1} we have
\begin{eqnarray*}
\widehat{MC}\Big( (\nabla^\perp \Psi ,a) ,   (e_{1~m_2},b) \Big) &=&   K C^2>0.
\end{eqnarray*}
Finally, for a suitable choice of the parameters $K$ and $A$, the index
\begin{eqnarray*}
\widehat{MC}\Big( (\nabla^\perp \Psi ,a) ,   (e_{l_2m_2},b) \Big) &=&   |A|^2{MC}(  e_{l_1m_1} , e_{l_2m_2} ) \\
&&   + m_2^2C^2(K+ 2 -l_2(l_2+1) )
\end{eqnarray*}
could be positive for any $0<|m_2|\leq l_2$ (i.e., $K$ large enough and $|A|^2$ small). 
We see that the Coriolis effect makes the system more stable and creates conjugate points that wouldn't exist without it.

The same argument applies analogously to the case where $l_1\geq 3$ and $2\leq l_2<l_1$, implying that $MC(e_{l_1 1} ,e_{l_2~1})>0$.

\end{cor}
%
%
%
%%%%%%%%%%%%%%%%%%%%%%%%%%%%%%%%%%%%%%%%%%%%%%%%%%%%%   Appendix  m=2     %%%%%%%%%%%%%%%%%%%%%%%%%%%%%%%%%%%%%%%%%%%%%%%%%%%%%%%%
%%%%%%%%%%%%%%%%%%%%%%%%%%%%%%%%%%%%%%%%%%%%%%%%%%%%%   Appendix  m=2     %%%%%%%%%%%%%%%%%%%%%%%%%%%%%%%%%%%%%%%%%%%%%%%%%%%%%%%%
%%%%%%%%%%%%%%%%%%%%%%%%%%%%%%%%%%%%%%%%%%%%%%%%%%%%%   Appendix  m=2     %%%%%%%%%%%%%%%%%%%%%%%%%%%%%%%%%%%%%%%%%%%%%%%%%%%%%%%%
%
%
%

\textbf{Acknowledgements}. Funded by the Deutsche Forschungsgemeinschaft (DFG, German
Research Foundation) - 517512794. 
%The author expresses gratitude to Prof. George Savvidy, Prof. Tsuyoshi Yoneda and Taito Tauchi for their valuable contributions during the %preparation of this paper.
%
%
%
%
%%%%%%%%%%%%%%%%%%%%%%%%%%%%%%%%%%%%%%%%%%%%%%%%%%%%%%%          bliography      %%%%%%%%%%%%%%%%%%%%%%%%%%%%%%%%%%%%%%%%%%%%%%%%%%%%%%%%%%%%%%%%%%
%%%%%%%%%%%%%%%%%%%%%%%%%%%%%%%%%%%%%%%%%%%%%%%%%%%%%%%          bliography      %%%%%%%%%%%%%%%%%%%%%%%%%%%%%%%%%%%%%%%%%%%%%%%%%%%%%%%%%%%%%%%%%%
%%%%%%%%%%%%%%%%%%%%%%%%%%%%%%%%%%%%%%%%%%%%%%%%%%%%%%%          bliography      %%%%%%%%%%%%%%%%%%%%%%%%%%%%%%%%%%%%%%%%%%%%%%%%%%%%%%%%%%%%%%%%%%
%
%

\end{document}